\documentclass{amsart}
\usepackage{mathptmx}

\usepackage{amssymb}
\usepackage{amsmath}

 \newtheorem{theo}{Theorem}[section]
 \newtheorem{lemm}[theo]{Lemma}

\newtheorem{prop}[theo]{Proposition}
\newtheorem{rema}[theo]{Remark}

\newcommand\NN{{\mathbb N}}
\newcommand\RR{{{\mathbb R}}}
\newcommand\ZZ{{\mathbb Z}}

\newcommand\TT{{{\mathbb T}}}
\def\SS {\mathbb{S}}

\def\la{\langle}
\def\ra{\rangle}

\let\a=\alpha
\let\b=\beta

\newcommand\cD{\mathcal D}
\newcommand\cE{{\mathcal E}}

\newcommand\cS{{\mathcal S}}
\newcommand\cM{{\mathcal M}}
\newcommand\cT{{\mathcal T}}

\newcommand\pa{\partial}

\begin{document}
\title[Bounded Solutions]
{Bounded  Solutions of the Boltzmann Equation\\
in the Whole Space}
\author{R. Alexandre}
\address{R. Alexandre,
\newline\indent
Department of Mathematics, Shanghai Jiao Tong University
\newline\indent
Shanghai, 200240, P. R. China,
and \newline\indent
IRENAV Research Institute, French Naval Academy
Brest-Lanv\'eoc 29290, France
}
\email{radjesvarane.alexandre@ecole-navale.fr}
\author{Y. Morimoto }
\address{Y. Morimoto, Graduate School of Human and Environmental Studies,
Kyoto University
\newline\indent
Kyoto, 606-8501, Japan} \email{morimoto@math.h.kyoto-u.ac.jp}
\author{S. Ukai}
\address{S. Ukai, 17-26 Iwasaki-cho, Hodogaya-ku, Yokohama 240-0015, Japan}
\email{ukai@kurims.kyoto-u.ac.jp}
\author{C.-J. Xu}
\address{C.-J. Xu, School of Mathematics, Wuhan University 430072,
Wuhan, P. R. China
\newline\indent
and \newline\indent
Universit\'e de Rouen, UMR 6085-CNRS,
Math\'ematiques
\newline\indent
Avenue de l'Universit\'e,\,\, BP.12, 76801 Saint
Etienne du Rouvray, France } \email{Chao-Jiang.Xu@univ-rouen.fr}
\author{T. Yang}
\address{T. Yang, Department of mathematics, City University of Hong Kong,
Hong Kong, P. R. China
\newline\indent
and \newline\indent
School of Mathematics, Wuhan University 430072,
Wuhan, P. R. China} \email{matyang@cityu.edu.hk}

\subjclass[2000]{35A01, 35A02, 35A09, 35S05, 76P05, 82C40}

\keywords{Boltzmann equation, non-cutoff cross section,
local existence, locally uniform Sobolev space, spatial behavior at infinity, 
pseudo-differential calculus.}

\maketitle

\date{empty}

\begin{abstract}
We construct bounded classical solutions of the Boltzmann equation
in the whole space  without  specifying
any limit behaviors at the spatial infinity and without assuming the smallness condition on initial data.
More precisely, we show that if the initial data
is non-negative and belongs to a uniformly local Sobolev space in the space variable 
 with Maxwellian type decay property in the velocity variable, 
then the Cauchy problem of the Boltzmann
equation possesses
a unique non-negative local solution in the same function
space,
 both for the cutoff and non-cutoff collision cross section  with
mild singularity.
The known solutions such as solutions on the torus (space periodic solutions) and in the vacuum (solutions
vanishing at the spatial infinity), and solutions in the whole space having a limit
equilibrium state at the spatial infinity
  are included in our category.
\end{abstract}


\section{Introduction}\label{s1}

Consider the Boltzmann
equation,
\begin{equation}\label{1.1}
f_t+v\cdot\nabla_x f=Q(f, f),
\end{equation}
where $f= f(t,x,v)$ is the density distribution function
of particles  with position
$x\in \RR^3$ and velocity $v\in \RR^3$ at time $t$.
The right hand side of (\ref{1.1}) is given by the
Boltzmann bilinear collision operator
\[
Q(g, f)=\int_{\RR^3}\int_{\mathbb S^{2}}B\left({v-v_*},\sigma
\right)
 \left\{g(v'_*) f(v')-g(v_*)f(v)\right\}d\sigma dv_*\,,
\]
which is well-defined for
 suitable functions $f$ and $g$ specified later. Notice that the
collision operator $Q(\cdot\,,\,\cdot)$ acts only on the velocity
variable $v\in\RR^3$. In the following discussion, we will use the
$\sigma-$representation, that is, for $\sigma\in \mathbb S^{2}$,
$$
v'=\frac{v+v_*}{2}+\frac{|v-v_*|}{2}\sigma,\,\,\, v'_*
=\frac{v+v_*}{2}-\frac{|v-v_*|}{2}\sigma,\,
$$
which give the relations between the pre- and post- collisional velocities.

It is well known that the Boltzmann equation is a fundamental
equation in statistical physics. For the mathematical theories on
this equation, we refer the readers to
\cite{Ce88,Cercignani-Illner-Pulvirenti,D-L,Grad,villani2}, and
the references therein also for the physics background.

In addition to the special bilinear structure of the collision
operator, the cross-section $B(v-v_*,\sigma)$
is a function of  only $|v-v_*|$ and $\theta$ where
\[
\cos\theta=\big<\frac{v-v_*}{|v-v_*|}, \sigma\big>, \,\quad
0\leq\theta\leq\frac{\pi}{2}.
\]
$B$ varies with different
physical assumptions on the particle interactions and it plays an
important role in the well-posedness theory for the Boltzmann equation.
In fact, except for the hard sphere model, for most of the other molecular
interaction potentials such as  the inverse power laws, the cross
section  $B(v-v_*,\sigma)$ has a non-integrable angular singularity.
For example, if the interaction potential obeys the inverse power law
 $r^{-(p-1)}$ for $2<
p<\infty$, where $r$ denotes the distance between
 two interacting  molecules,
the cross-section behaves like
\[
B(|v-v_*|, \sigma)\sim |v-v_*|^\gamma \theta^{-2-2s},
\]
with
\[
-3<\gamma=\frac{p-5}{p-1}<1,\,\,\, \,\,\,\,\,\,0<s=\frac{1}{p-1}<1\, .
\]
As usual, the hard and soft potentials correspond to $p>5$ and
$2<p<5$  respectively, and the Maxwellian potential corresponds to
$p=5$.

The  main consequence of the non-integrable singularity of
$B$ at $\theta=0$  is that it makes the collision operator $Q$ behave like
a pseudo differential operator, as
 pointed out by many authors, e.g. \cite{al-1,  Lions98, Pao, ukai}.
 To avoid this difficulty, Grad \cite{Grad} introduced
 an assumption to cutoff this singularity.
This was a substantial step made in the study of the Boltzmann
equation \eqref{1.1} and is now called Grad's angular cutoff assumption.

One of the main issues in the study of \eqref{1.1} is the existence theory of the solutions.
Many authors have developed various  methods for constructing
local and global solutions for different situations. Among them, the Cauchy problem has been studied most extensively for both cutoff and non-cutoff cases. 

An essential observation here is 
  that so far, all  solutions  for the Cauchy problem   have been constructed so as to 
satisfy  one of the following three
spatial behaviors at infinity;
$x$-periodic solutions (solutions on the torus,
\cite{ Grad, gr-st-2, ukai-1a}), solutions approaching an equilibrium
(\cite{amuxy3,amuxy3b,amuxy4-2,amuxy4-3, guo-1, liu-2, ukai-1b})
and solutions approaching 0 (solutions near vacuum, \cite{alex-two-maxwellian,amuxy-arma,al-3,
D-L}).
Notice that the solutions constructed  in \cite{guo-bounded} are also  solutions approaching an global equilibrium.   

However, it is natural to wonder if there are any other solutions 
behaving differently at $x$-infinity.
In fact the aim of the present paper is to show that the admissible limit behaviors are
not restricted to the above three behaviors. More precisely, 
we show that the Cauchy problem admits a very large solution space 
which includes not only the solutions of the above three types but also
the solutions having no specific limit behaviors such that  almost periodic solutions and
perturbative solutions of arbitrary bounded functions which are not necessarily equilibrium state.
This will be done for both cutoff and non-cutoff cases without the smallness condition
on initial data. 

The method developed in this paper  works for local existence theory. The global existence
in the same solution space is a big open issue and is  our current subject.
Also the present  method  works for the Landau equation but
since the extension is straightforward, the detail is omitted.

Our assumption on the cross section is as follows.
For the non-cutoff case  we assume as usual that $B$ takes the form
\begin{align}\label{B}
B(v-v_*,\sigma)=\Phi(|v-v_*|)b(\cos\theta),
\end{align}
in which it contains a kinetic part
\[
\Phi(|v-v_*|)=\Phi_\gamma(|v-v_*|)=|v-v_*|^\gamma,
\]
and a factor related to the collision angle with singularity,
\[
b(\cos\theta)\approx K\theta^{-2-2s}  \text{\ \ when\ \ } \theta\to 0+,
\]
for some constant $K>0$. For the cutoff case we assume that $b$ takes the form
\eqref{cutoff} or is bounded by it.

In order to introduce our working  function spaces, set
\[
\pa^\alpha_\beta=\pa^\alpha_x\pa^{\beta}_v, \qquad \alpha, \beta \in \NN^3.
\]
Let 
 $\phi_1=\phi_1(x)$ be a smooth cutoff function
\begin{align}\label{phi1}
\phi_1\in C^\infty_0(\RR^3), \quad 0\le \phi_1(x)\le 1, \quad \phi_1(x)=\left\{
\begin{array}{ll}
1, \ &|x|\le 1,
\\
0,  &|x|\ge 2.
\end{array}
\right.
\end{align}
Our  function space is  the uniformly local Sobolev space with respect to the space variable
and the usual Sobolev space with respect to the velocity variable with weight. 
More precisely, let $k, \ell \in\NN $
and $W_\ell=(1+|v|^2)^{\ell/2}$ be a weight function. We define
\begin{align}\label{ulS}
H^{k,\ell}_{\it ul}&(\RR^6)=\{g \ |\  \|g\|_{H^{k,\ell}_{\it ul}(\RR^6)}^2
\\&=\notag
\sum_{|\alpha+\beta|\le k}\ \sup_{a\in \RR^3}\int_{\RR^6} |\phi_1(x-a) W_\ell\pa^\alpha_\beta g(x,v)|^2dxdv<+\infty\}.
\end{align}
We will set $H^{k}_{\it ul}(\RR^6)=H^{k,0}_{\it ul}(\RR^6)$.

 The uniformly local Sobolev space was first introduced by Kato in \cite{kato}
as a space of functions of $x$ variable, and was used to develop
the local existence theory on
 the quasi-linear symmetric hyperbolic systems without specifying the limit behavior at infinity.

This space could be defined also by the cutoff function  $\phi_R(x)=\phi(x/R)$
for any $R>0$, but the choice of $R$ is not a matter. Indeed  let $R>1$. Then
we see that
\begin{align*}
\|g\|_{H^{k,\ell}_{\it ul}(\RR^6)}&
\le \sum_{|\alpha+\beta|\le k}\
\sup_{a\in\RR^3}\int_{\RR^6}|\phi_R(x-a)W_\ell\pa^\alpha_\beta g(x,v)|^2dxdv
\\&\le \notag
\sum_{|\alpha+\beta|\le k}\ 
\sum_{j\in \ZZ^3, |j|\le R}
\sup_{a\in\RR^3}\int_{\RR^6}|\phi_1(x-a-j)W_\ell\pa^\alpha_\beta g(x,v)|^2dxdv
\\&\le \notag
C R^3\|g\|_{H^{k,\ell}_{\it ul}(\RR^6)}.
\end{align*}
The case $0<R<1$ can be proved similarly. In the sequel, therefore, we fix $R=1$.

This space shares many important properties with the usual Sobolev space
such as the Sobolev embedding and hence  it is contained in the space of bounded functions
if $k>3$.
An important difference from the usual Sobolev space is that
no limit property is specified at $x$-infinity for the space \eqref{ulS}.

We shall consider the solutions satisfying
the Maxwellian type exponential decay in the velocity variable.
More precisely, set $\la v\ra=(1+|v|^2)^{1/2}$. For $k\in\NN$,  
our function space of initial data will be
$$
\cE^k_0(\RR^6)=\Big\{g\in\cD'(\RR^6_{x, v});\, \exists \,\rho_0>0\,
\, s. t. \,\, e^{\rho_0 <v>^2} g\in H^{k}_{\it ul}(\RR^6_{x, v}) \Big\},
$$
while the function space of solutions will be, for $T>0$,
\begin{eqnarray*}
{\mathcal E}^k([0,T]\times{\mathbb R}^6_{x, v})&=&\Big\{f\in
C^0([0,T];{\mathcal D}'({\mathbb R}^6_{x, v}));\, \exists \,\rho>0
\\
&&\hskip 0.5cm s. t. \,\, e^{\rho \langle v \rangle^2} f\in C^0([0,
T];\,\, H^{k}_{\it ul}({\mathbb R}^6_{x, v})) \Big\}.
\end{eqnarray*}
Our main result is stated as follows.
\begin{theo}\label{E-theo1}
 Assume that the cross section $B$ takes the form \eqref{B} with
$0<s<1/2$, $\gamma>-3/2$ and  $ 2s+\gamma< 1$.
If the initial data $f_0$ is non-negative
and belongs to the function space
$
\cE^{k_0}_0(\RR^6)$ for some  $ k_0 \in \NN, k_0\ge 4$, then, there exists
$T_* >0$ such that the  Cauchy
problem
\begin{equation}\label{1.1b}
\left\{\begin{array}{l}f_t+v\cdot\nabla_x f=Q(f, f),\\
f|_{t=0}=f_0,\end{array}\right.
\end{equation}
admits a non-negative  unique solution in the function space
${\mathcal E}^{k_0}([0,T_*]\times{\mathbb R}^6)$.
\end{theo}
 \begin{rema}\label{rem13}
For the cutoff case, if
$\gamma>-3/2$, the same theorem and the same proof are valid
because  our assumption is that the cross-section $b$ is given by \eqref{cutoff} or is bounded by it.
In the sequel,  therefore, we
consider  the non-cutoff case only.
\end{rema}

Before closing this section we give some comparison of  this paper and our recent paper
\cite{amuxy-arma}. First, we shall compare the existence results.
Both papers solve the same modified Cauchy problem
\eqref{E-Cauchy-B}.
Thus, we shall compare Theorem \ref{E-theo-0.2} of this
 paper 
and Theorem 4.1 of \cite{amuxy-arma}. 

 The solution space in 
\cite{amuxy-arma} is the usual weighted Sobolev space $H^k_\ell(\RR^6_{x,v})$ , 
so that  Theorem 4.1 of \cite{amuxy-arma} gives 
solutions vanishing at $x$-infinity (solutions near vacuum).  And it is easy to see
that even if 
 the space is replaced by $H^k_\ell(\TT^3_x\times \RR^3_v)$,
the proof of \cite{amuxy-arma} is still
valid and gives rise to $x$-periodic solutions (solutions on torus). The same space was used also
in \cite{gr-st-2}. 
On the other hand, it is clear that 
the space $H^{k,\ell}_{\it ul}(\RR^6)$
defined by \eqref{ulS}, the locally uniform Sobolev space with respect to $x$-variables, 
contains, as its subset,
not only the spaces $H^k_\ell(\RR^6)$ and $H^k_\ell(\TT^3\times\RR^3)$ 
but also the
set of functions having the form $G=\mu+\mu^{1/2}g$. If $\mu$ is a
global Maxwellian, then we have well-known
perturbative solutions of equilibrium
as in \cite{amuxy3b, amuxy4-2, amuxy4-3, ukai-1b}, but more generally
$G$ can be any bounded functions. Hence, Theorem \ref{E-theo-0.2}
gives, for example, almost periodic solutions, solutions having different limits at $x$-
infinity like shock profile solutions which attain different equilibrium at the right and left infinity,
and bounded solution behaving in more general way at $x$-infinity.
Thus the present paper extends extensively the function space of admissible solutions. This is an
essential difference between the two papers.
 
Another big difference is the collision cross section. The present paper deals with
the original kinetic factor $\Phi(z)=|z|^\gamma$ without regularizing the
singularity at $z=0$ whereas   \cite{amuxy-arma} considers only a
regularized one of the form $\Phi(z)=(1+|v|^2)^{\gamma/2}$. This regularization
simplifies drastically the estimates of collision operator $Q$. 
 For example the proof of Proposition 4.4 with non-regularized kinetic factor is 
far more subtle than (2.1.2) of \cite{amuxy-arma}.  Also the case $\gamma<0$
should be handled separately from the case $\gamma\ge 0$
if the  kinetic factor is not regularized, and the range of admissible values of $\gamma$
is restricted in this paper.

The rest of this paper is organized as follows.
In the next section we first rewrite the Cauchy problem
\eqref{1.1b} by the one involving the wight function of the
time-dependent Maxwellian type and approximate it by introducing the cutoff cross section.
After establishing the upper bounds of the cutoff collision operator,
we introduce  linear iterative Cauchy problems and show that the
iterative solutions converge to the solutions to
the cutoff  Cauchy problem.
In Section 3 we derive the a priori  estimates satisfied by
the solutions to the cutoff Cauchy problem uniformly with respect to the cutoff
parameter.  The estimates thus obtained are enough to
conclude the local existence and hence to lead to Theorem \ref{E-theo1}.
The last section is devoted to the proof of Lemma \ref{lem410} which is
essential for the uniform estimate established in Section 3.

\section{Construction of Approximate Solutions}\label{section2}
\setcounter{equation}{0}

As will be seen later (Lemma \ref{E-lemm3.111} and Theorem \ref{upperQ}), the non-linear collision operator
induces a weight loss, which implies that it cannot be Lipschitz continuous
so that the usual iteration procedure is not valid for constructing
local solutions.
This difficulty can be overcome, however, by introducing
 weight functions in $v$ of time-dependent Maxwellian type, developed previously in
\cite{amuxy-arma, ukai, ukai-2}. Indeed it compensates the weight loss  by
producing
an extra gain term of one order higher weight in the velocity variable at the
expense of the loss of the decay order of the time dependent Maxwellian-type weight.
\subsection{Modified Cauchy Problem}\label{E-s0} \setcounter{equation}{0}
More precisely,
we set, for any $\kappa, \rho>0$,
\[
T_0=\rho/(2\kappa),
\]
and put
\[
\mu_\kappa(t)=\mu(t,v)=e^{-(\rho-\kappa t)(1+ |v|^2)},
\]
and
$$
f=\mu_\kappa(t) g, \,\,\, \quad \Gamma^t(g, g)=\mu_\kappa(t)^{-1}Q(\mu_\kappa(t) g,\,
\mu_\kappa(t) g)
$$
for $t\in [0,T_0]$.
Then the Cauchy problem (\ref{1.1b}) is reduced to
\begin{equation}\label{E-Cauchy-B}
\left\{\begin{array}{l}
g_t+v\cdot\nabla_x g\  +\kappa (1+ |v|^2) g=\Gamma^t(g, g),\\
g|_{t=0}=g_0.
\end{array}\right.
\end{equation}

Define
\begin{align*}
 \cM^{k, \ell}&(]0,T[\times{\mathbb R}^6))
=\{
 g \ | \ \|g\|_{ \cM^{k, \ell}(]0,T[\times{\mathbb R}^6))}
 \\&=
 \sum_{|\alpha+\beta|\le k}\sup_{a\in\RR^3}
\int_{]0,T[\times \RR^6} |\phi_0(x-a) W_\ell\pa^\alpha_\beta g(x,v)|^2dtdxdv<+\infty\}.
\end{align*}
Our existence theorem can be stated as follows

\begin{theo}\label{E-theo-0.2}
Assume that  $0<s<1/2,\, \gamma >-3/2$ and $ 2s+\gamma<1$. Let $\kappa, \rho > 0$ and  let
$
g_0 \in H^{k,\ell}_{\it ul} ({\mathbb R}^6 )$, $g_0\geq 0$ for some $k\geq
{4}$ and $\ell\geq 3$. Then there exists $T_* \in ]0, T_0]$ such that the Cauchy problem
(\ref{E-Cauchy-B}) admits a unique  non-negative solution satisfying
\[
g \in C^0 ([0,T_*];\,\, H^{k,\ell}_{\it ul}({\mathbb R}^6))\bigcap\,\,
 \cM^{k, \ell+1}(]0,T_*[\times{\mathbb R}^6))\,. 
\]
\end{theo}

 The strategy of proof
is
in the same spirit as in \cite{amuxy-arma}. That is, first,
 approximate the non-cutoff cross-section by a family of
cutoff cross-sections and construct the corresponding solutions
by a sequence of iterative linear equations. Then the existence of solutions to these approximate linear equations and by obtaining a
uniform estimate on these solutions
 with respect to the cutoff parameter
in the uniformly local Sobolev space, the compactness argument will
lead to the convergence of the approximate
solutions to the desired solution for the original problem.

\subsection{Cutoff Approximation}

Recall that the cross-section takes the form \eqref{B}.
For $0<\varepsilon<\,<1$,
 we approximate (cutoff)  the cross-section by
\begin{align}\label{cutoff}
b_\varepsilon(\cos\theta)=\left\{\begin{array}{l} b(\cos\theta),
\,\,\,\,\,\mbox{if}\,\,\,|\theta|\geq 2\varepsilon,\\
b(\cos\varepsilon), \,\,\,\,\,\mbox{if}\,\,\,|\theta|\leq
2\varepsilon.
\end{array}\right.
\end{align}
Denote  the corresponding cutoff cross-section by $B_\varepsilon=\Phi(v-v_*)b_\varepsilon(\cos\theta)$
and the collision operator by $\Gamma^t_\varepsilon(g,\,g)$.
We shall establish a
 upper weighted estimate on the cutoff  collision operator
in the uniformly local Sobolev space $H^{k,\ell}_{\it ul}(\RR^6).$
\begin{lemm}\label{E-lemm3.111}
Let  $\, \gamma>-3/2 $.
Then for any $\varepsilon>0,\, k\geq 4,\, l\geq 0$,
and
for any $U,V$
belonging to $H^{k,\ell}_{\it ul}(\RR^6)$, it holds that
\[
\Gamma_\varepsilon^t(U,\, V)\in H^{k,\ell}_{\it ul}(\RR^6)
\]
with
\begin{equation}\label{E-uper-estimate}
\|\Gamma_\varepsilon^t(U,\, V)\|_{H^{k,\ell}_{\it ul}(\RR^6)}\leq C
\|U\|_{H^{k,\ell+\gamma^+}_{\it ul}(\RR^6)} \|V\|_{H^{k,\ell+\gamma^+}_{\it ul}(\RR^6)},
\quad 0\le t\le T_0,
\end{equation}
for some
$C>0$ depending  on $\varepsilon,\, k,\, \ell$ as well as $\rho$, $\kappa$.
\end{lemm}
\begin{proof}
First, for simplicity of  notations,  denote $\mu_\kappa(t)$ by $\mu(t)$
without any confusion. By using the collisional energy conservation,
$$
|v'_*|^2+|v'|^2=|v_*|^2+|v|^2,
$$
we have $\mu_*(t)=\mu^{-1}(t)\,\mu'_*(t)\,\mu'(t)$. Then for
suitable functions $U, V$, it holds that
\begin{align}\notag
\Gamma_\varepsilon^t&(U,\,
V)(v)
\\&=\notag
\mu^{-1}(t,v)\iint_{\RR^3_{v_*}\times\mathbb S^{2}_{\sigma}}
B_\varepsilon(v-v_*,\, \sigma) \big(\mu'_*(t) U'_* \mu'(t) V'-
\mu_*(t) U_*
\mu(t) V\big) d v_* d \sigma
\\
&=\iint_{\RR^3_{v_*}\times\mathbb S^{2}_{\sigma}}
B_\varepsilon(v-v_*,\, \sigma)\mu_*(t)\, \big( U'_*  V'-  U_* V\big)
d v_* d \sigma\label{E-3.105}
\\&={\cT}_\varepsilon(U,\, V,\, \mu(t))\notag.
\end{align}
Then we have by the Leibniz formula
in the $x$ variable and by the translation invariance property in the $v$
variable that for any $\alpha, \beta\in\NN^3$,
\begin{align*}
\partial^\alpha_\beta&\Gamma_\varepsilon^t(U,\,\,
V)
 =\sum_{\begin{subarray}{l}\alpha_1+\alpha_2=\alpha;\,
\\
\beta_1+\beta_2+\beta_3=\beta
\end{subarray}}
C_{\alpha_1, \alpha_2, \beta_1,
\beta_2,
\beta_3}{\cT}_\varepsilon(\partial^{\alpha_1}_x\partial^{\beta_1}_v
U,\,\,\partial^{\alpha_2}_x\partial^{\beta_2}_v V,\,\,
\partial^{\beta_3}_v\mu(t)
).
\end{align*}

Next, recall the  cutoff function $\phi_1$ in \eqref{phi1} and set $\phi_2(x)=\phi_1(x/2)$, that is,
\begin{align*}
\phi_2\in C^\infty_0(\RR^3), \quad 0\le \phi_2(x)\le 1, \quad \phi_2(x)=\left\{
\begin{array}{ll}
1, \ &|x|\le 2,
\\
0,  &|x|\ge 4.
\end{array}
\right.
\end{align*}
Since $\phi_1(x-a)=\phi_1(x-a)\phi_2(x-a)$ holds, we see that for $a\in\RR^3$,
\begin{align*}
\phi_1&(x-a)\partial^\alpha_\beta\Gamma_\varepsilon^t(U,\,\,
V)
 \\&=\sum_{\begin{subarray}{l}\alpha_1+\alpha_2=\alpha;\,
\\
\beta_1+\beta_2+\beta_3=\beta
\end{subarray}}
C_{\alpha_1, \alpha_2, \beta_1,
\beta_2,
\beta_3}{\cT}_\varepsilon(\phi_1(x-a)\partial^{\alpha_1}_{\beta_1}
U,\,\,\phi_2(x-a)\partial^{\alpha_2}_{\beta_2} V,\,\,
\partial^{\beta_3}_v\mu(t)
).\notag
\end{align*}
To prove \eqref{E-uper-estimate}, put
\begin{align*}
&g_1=\phi_1(x-a)\partial^{\alpha_1}_{\beta_1} U,\,\, \qquad
h_2=\phi_2(x-a)\partial^{\alpha_2}_{\beta_2} V,\,\, \qquad
\mu_3(t)=
\partial^{\beta_3}_v\mu(t),
\\&
{\cT}_\varepsilon(g_1,h_2,\mu_3(t))={\cT}_\varepsilon^+
-{\cT}_\varepsilon^-.
\end{align*}
Throughout this section, we often use the estimates
\begin{align*}
\mu(t,v),\quad |\mu_3(t)|= |\partial^{\beta_3}_v\mu(t,v)|\le
C_{\rho,\, k} \,\,e^{-\rho \langle v\rangle ^2/4}, \qquad t\in[0,
T_0], \quad v\in\RR^3.
\end{align*}
We compute ${\cT}_\varepsilon^+$ as follows.
\begin{align*}
|W_\ell {\cT}_\varepsilon^+|&\le C \iint  |v-v_*|^\gamma
|\mu_3(t,v_*)|\frac{W_\ell}{(W_{\ell})'_*(W_{\ell})'}|(W_{\ell}g_1)'_*|| (
W_{\ell}h_2)'|dv_*d\sigma
\\&
\le C\Big[ \iint |v-v_*|^{2\gamma}
 |\mu_3(t,v_*)|^2dv_*d\sigma
 \Big]^{1/2}
 \\&\hspace*{2cm}\times\Big[ \iint
\ |(W_{\ell}g_1)'_*
 (W_{\ell}h_2)'|^2dv_*d\sigma \Big]^{1/2}
 \\&\le
 C\Big[ \iint W_{2\gamma}|(W_{\ell}g_1)'_*
 (W_{\ell}h_2)'|^2dv_*d\sigma \Big]^{1/2},
 \end{align*}
where we have used
$$
\frac{W_l}{(W_l)'_*(W_l)'}\leq 1$$
which comes from the energy conservation, and an elementary inequality
\begin{align}\label{ele}
\iint |v-&v_*|^{\gamma}
 |\mu_3(t,v_*)|dv_*d\sigma
\\&\le C\notag
 \iint |v-v_*|^{\gamma} e^{-\rho|v_*|^2/4}dv_*
\le C\langle v\rangle^{\gamma}, \quad \gamma>-3,
 \end{align}
 with some $\rho>0$.
 Since the change of variables
\begin{equation*}
(v,v_*,\sigma) \to (v',v'_*,\sigma'),\qquad \sigma'=(v-v_*)/|v-v_*|,
\end{equation*}
has a unit Jacobian, and since $W_{\gamma}\le (W_{\gamma^+})'_*(W_{\gamma^+})'$ holds,
we get
\begin{align*}
\|W_\ell {\cT}_\varepsilon^+\|_{L^2(\RR^6)}^2 &\le
C\iiiint |(W_{\ell+\gamma^+}g_1)'_*
 ( W_{\ell+\gamma^+} h_2)'|^2dv_*d\sigma dvdx
 \\&
 \le C\int\|W_{\ell+\gamma^+}g_1\|^2_{L^2(\RR^3_v)} \|
W_{\ell+\gamma^+} h_2\|^2_{L^2(\RR^3_v)}dx.
\end{align*}
If $|\a_1 + \b_1| \leq 2$,  by virtue of  the  Sobolev embedding and by the assumption
$k \ge 4$,  we have
\begin{align*}
\|W_\ell{\cT}_\varepsilon^+\|_{L^2(\RR^6)}
&\le C \|W_{\ell+\gamma^+}g_1\|_{L^\infty(\RR_x^3 ; L^2(\RR^3_v))} \|
W_{\ell+\gamma^+} h_2\|_{L^2(\RR^6_{x,v})}\\
&\leq C
\|\phi_1(x-a)W_{\ell+\gamma^+}\partial^{\alpha_1}_{\beta_1}
U\|_{H^2_x(L^2_v)} \|\phi_2(x-a)W_{\ell+\gamma^+}\partial^{\alpha_2}_{\beta_2}V\|_{L^2(\RR^6_{x,v})}.
\end{align*}
 Taking the supremum of both sides with respect to  $a\in \RR^3$ and since
$\phi_1$ and $\phi_2$ define the equivalent norms, we have
\[
\|{\cT}_\varepsilon^+\|_{H^{0, \ell}_{\it ul}(\RR^6)}
\le C \|U\|_{H^{k,\ell+\gamma^+}_{\it ul}(\RR^6)}\|V\|_{H^{k,\ell+\gamma^+}_{\it ul}(\RR^6)}.
\]
The computation is similar when $|\a_2 + \b_2| >2$ for which $|\a_1+\b_1|\le k- |\a_2 + \b_2|
\le k-3$. Also, the the estimate of $\cT^-_\varepsilon$
can be done similarly but more straightforwardly.
This completes  the proof of the lemma.
 In the below we will use the following estimates which comes directly from the above proof.\end{proof}

\begin{lemm}\label{E-lemm3.111a}
Let  $\,  \gamma>-3/2 $.
Then for any $\varepsilon>0,\, k\geq 4,\, l\geq 0$,
and
for any $U,V$
belonging to $H^{k,\ell}_{\it ul}(\RR^6)$, it holds that
\begin{align}\label{E-uper-estimate}
\|\phi_1(x&-a)W_{\ell}\partial^{\alpha}_{\beta} \Gamma_\varepsilon^{t,\pm}(U,\, V)\|_{L^2(\RR^6)}
\\&\leq C\sum\Big(
 \|\phi_1(x-a)W_{\ell+\gamma^+}\partial^{\alpha_1}_{\beta_1}
 U\|_{H^2_x(L^2_v)}
\|\phi_2(x-a)W_{\ell+\gamma^+}\partial^{\alpha_2}_{\beta_2}V\|\notag
\\&\quad +
\|\phi_1(x-a)W_{\ell+\gamma^+}\partial^{\alpha_1}_{\beta_1}
 U\|_{L^2(\RR^6)}
\|\phi_2(x-a)W_{\ell+\gamma^+}\partial^{\alpha_2}_{\beta_2}V\|_{H^2_x(L^2_v)}\Big)
.\notag
\end{align}
for some
$C>0$ depending  on $\varepsilon,\, k,\, \ell$ as well as $\rho$, $\kappa$.
\end{lemm}

We now study the following Cauchy problem for the cutoff Boltzmann
equation

\begin{equation}\label{E-Cauchy-cut-off}
\left\{\begin{array}{l}
g_t+v\cdot\nabla_x g +\kappa \la v \ra^2 g=\Gamma^t_\varepsilon(g, g),\\
g|_{t=0}=g_0\,,
\end{array}\right.
\end{equation}
for which we shall obtain uniform estimates in weighted
Sobolev spaces.
We first prove the existence of weak solutions.

\begin{theo}\label{E-exist-cut-off}
Assume that $-3/2<\gamma\le 1$. Let $k\ge  4,\, l \ge 0$, $\varepsilon>0$
and $D_0>0$. Then, there exists $T_\varepsilon \in ]0, T_0]$
such that for any    initial
data $g_0$ satisfying
\begin{equation}\label{initial}
 g_0\in H^{k,\ell}_{\it ul}({\mathbb R}^6),  \qquad g_0\ge 0,
\qquad \|g_0\|_{H^{k,\ell}_{\it ul}({\mathbb
R}^6)}\le D_0,
\end{equation}
the Cauchy problem (\ref{E-Cauchy-cut-off}) admits a unique
solution $g^\varepsilon$ having the property
$$
 g^\varepsilon\in C^0(]0, T_\varepsilon[;\,\, H^{k,\ell}_{\it ul}({\mathbb R}^6)) ,
 \qquad g\ge 0,
\qquad \|g^\varepsilon\|_{L^\infty(]0, T_\varepsilon[;\,\,
H^{k,\ell}_{\it ul}({\mathbb R}^6))}\le 2D_0.
$$
Moreover, this solution enjoys a moment gain in the sense that
\begin{equation*}
g^\varepsilon\in \cM^{k, \ell+1}(]0,T_\varepsilon[\times (\RR^6)).
\end{equation*}
\end{theo}
 \begin{rema}
{\rm (1) } Notice that we do not assume $g_0\in H^{k,\ell+1}_{\it ul}(\RR^6)$.
The moment gain will be  essentially used  below to control
the weight loss of the collision operator shown in Lemma \ref{E-lemm3.111}.

{\rm (2) } The regularity of $g^\varepsilon $ with respect to $t$
variable follows directly from the equation \eqref{E-Cauchy-cut-off}.

\end{rema}
\noindent {\em Proof of Theorem \ref{E-exist-cut-off}. } We prove
the existence of non-negative solutions by successive approximation
that preserves the  non-negativity, which is defined by using the usual
splitting of the collision operator \eqref{E-3.105} into the
gain (+) and loss (-) terms,
\begin{align*}
\Gamma^{t,+}_\varepsilon(g,h)&=\iint_{\RR^3_{v_*}\times\mathbb
S^{2}_{\sigma}} B_\varepsilon(v-v_*,\, \sigma)\mu_*(t)\,  g'_*  h' d
v_* d \sigma,
\\
\Gamma^{t,-}_\varepsilon(g,h)& = hL_\varepsilon(g),
\\
L_\varepsilon(g)&=\iint_{\RR^3_{v_*}\times\mathbb S^{2}_{\sigma}}
B_\varepsilon(v-v_*,\, \sigma)\mu(t,v_*)\,  g_* d v_* d \sigma.
\end{align*}
Evidently, Lemma \ref{E-lemm3.111} applies to
$\Gamma^{t,\pm}_\varepsilon$, and in view of \eqref{ele}, the
linear operator $L_\varepsilon$ satisfies
\begin{align}\notag
|\partial^{\alpha}_{\beta}&L_\varepsilon(g)(t,x,v)|\le
 C\sum_{\beta_1+\beta_2=\beta}\int |v-v_*|^\gamma
|(\pa_{\beta_1}\mu)(t,v_*)| \
|(\partial^{\alpha}_{\beta_2}g)(t,x,v_*)|dv_*
\\&\label{L}
\le C\Big[\int |v-v_*|^{2\gamma}
e^{-\rho|v*|^2/4}dv_*\Big]^{1/2}\sum_{\beta_2\le \beta}\Big[\int |(\partial^{\alpha}_{\beta_2}g)(t,x,v_*)|^2dv_*\Big]^{1/2}
\\&\notag
\le C
\langle v\rangle ^\gamma
\|\partial^{\alpha}_xg\|_{H^{|\beta|}(\RR_v^3)},
\quad t\in[0,T_0],
\end{align}
for a constant $C>0$ depending on $\varepsilon$.

We now define a sequence of approximate solutions
$\{g^n\}_{n\in\NN}$ by
\begin{equation}
\left\{
\begin{array}{l}
g^0=g_0\, ;\\
\partial_t g^{n+1}+v\cdot\nabla_x g^{n+1}+\kappa\langle |v|
\rangle^2 g^{n+1}
\\
\hspace{3cm}
=\Gamma_\varepsilon^{t,+} (g^n,\,
g^n)
-\Gamma_\varepsilon^{t,-} (g^n,\, g^{n+1}),\qquad
 \\
g^{n+1}|_{t=0}=g_0.
\end{array}\right.\label{iteration}
\end{equation}
Actually, in view of \eqref{L} we can consider the mild form
\begin{align}\label{mild}
g^{n+1}&(t,\, x,\, v)=e^{-\kappa\langle |v|\rangle^2 t-V^n(t,\,
0)}g_0(x-tv,\, v)
\\ \notag
&+\int^t_0
 e^{-\kappa\langle |v|\rangle^2(t-s)-V^n(t,\, s)}
\Gamma_\varepsilon^{s,\, +} (g^n,\, g^n)(s, x-(t-s)v, \, v) ds,
\end{align}
where
\[
V^n(t,\, s)=\int_s^t L_{\varepsilon}(g^n)(s, x-(t-s)v,\, v)ds.
\]

We shall first notice that \eqref{mild} defines well an approximate sequence of solutions.
In fact, it follows  from Lemma \ref{E-lemm3.111} that if
\begin{align}\label{indhypo}
&g_0\in H^{k,\ell}_{\it ul}(\RR^6), \quad g_0\ge 0,
\\&\notag
g^n\in L^\infty(]0, T[;\,\, H^{k,\ell}_{\it ul}(\RR^6)),\qquad g^n\ge 0,
\end{align}
for any $T\in ]0,T_0[$, then
 the mild form \eqref{mild} determines
$g^{n+1}$ in the function class
\begin{equation}\label{loss}
g^{n+1}\in L^\infty(]0, T[;\,\, H^{k,\ell-\gamma^+}_{\it ul}(\RR^6)),\qquad
g^{n+1}\ge 0,
\end{equation}
and solves \eqref{iteration}. Thus $g^{n+1}$ exists and is
non-negative. However, it  appears to have a loss of weight in the velocity
variable. We shall now show that
the term $\kappa\langle v \rangle^2 g^{n+1}$ in \eqref{iteration}
not only recovers this weight loss but also creates  higher moments.
 To see this,  introduce the space and
 norm defined by
 $$
 \begin{array}{ll}
 &Y=L^\infty(]0, T[;\,\, H^{k,\ell}_{\it ul}(\RR^6))\cap \cM^{k,\ell+1}(]0, T[\times \RR^6),
 \\[0.3cm]&
 ||g||_Y^2=\|g\|^2_{L^\infty(]0, T[;\,\, H^{k,\ell}_{\it ul}(\RR^6))}
 +\kappa\|g\|^2_{\cM^{k,\ell+1}(]0, T[\times \RR^6)}\,.
 \end{array}
 $$
\begin{lemm}\label{estimate1}
Assume that $ -3/2<\gamma\le 1$ and let $k\ge  4, l \ge 0, \varepsilon>0$.
Then, there exist positive numbers $ C_1,C_2$ such that if
$\rho>0,\, \kappa >0$ and if $g_0$ and $ g^n$ satisfy \eqref{indhypo}
with some $T\le T_0$, the function $g^{n+1}$ given by \eqref{mild}
enjoys the properties
\begin{align*}
&g^{n+1}\in Y
\\&
||g^{n+1}||_Y
\le e^{C_1 K_n T}
\left(\|g_0\|^2_{H^{k,\ell}_{\it ul}(\RR^6)}+\frac{C_2}{\kappa} ||g^n||^4_{L^4(]0,
T[;\,\, H^{k,\ell}_{\it ul}(\RR^6))}\right),
\end{align*}
where $K_n$ is a positive constant depending on
$\|g^n\|_{L^\infty(]0, T[;\,\, H^{k,\ell}_{\it ul}({\mathbb R}^6))}$ and $\kappa$.
\end{lemm}

\begin{proof}  Put
\begin{align}\label{hnl}
h^n_\ell
=\phi_1(x-a)W_\ell\partial^\alpha_\beta g^n.
\end{align}
Differentiate \eqref{iteration} with respect to $x,v$
and multiply by $\phi_1(x-a)$
to deduce
\begin{align*}
\partial_t h^{n+1}_\ell& +v\cdot\nabla_x h^{n+1}_\ell+
\kappa \langle  v \rangle ^2 h^{n+1}_\ell = G_1^+ -G_1^- +G_2 +G_3,
\\
&G_1^{+}=\phi_1(x-a)W_\ell\partial^\alpha_\beta \Gamma_\varepsilon^{t,+}(g^n,\,\,
g^n),
\\& G_1^{-}=\phi_1(x-a)W_\ell\partial^\alpha_\beta  \Gamma_\varepsilon^{t,-}(g^n,\,\,
g^{n+1}),
\\
&G_2=-W_\ell[\phi_1(x-a)\partial_\b ,\,\, v\cdot\nabla_x] \pa^\alpha g^{n+1} ,
\\
&G_3=-
\kappa W_\ell\sum_{|\beta'|=1,2}C_{\beta'}\phi_1(x-a)
( \partial_{\beta'}\langle v \rangle ^2)\partial^{\alpha}_{\beta -\beta'}
 g^{n+1}.
\end{align*}

Let
$\chi_j\in C^\infty_0(\RR^3), \,\, j\in\NN \,, $ be the cutoff functions
\[
\chi_j(v)=\left\{\begin{array}{ll} 1\, ,\,\,\,\,\,\,& |v|\le j\,,\\
0\,, & |v|\ge j+1\, . \end{array}\right.
\]
Let $S_N(D_x)$ be a mollifier definded by the Fourier multiplier
\[
S_N(\xi)=2^{-N}S(2^{-N}\xi), \quad S(\xi)\in \cS(\RR^3), \quad S(\xi)=1 \quad(|\xi|\le 1), \quad =0 \quad(|\xi|\ge 1).
\]
We remark that
\eqref{loss} does not necessarily imply $
h^{n+1}_{\ell+1}(t)\in L^2(\RR^6)$, but $\chi_j h^{n+1}_{\ell+1}(t)$
does for all $j\in \NN$. 
Hence, we can use  $\chi_j^2 S_N^2(D_x) h^{n+1}_\ell$ as a test function to get
\begin{align}\label{diff1}
&\frac{1}{2}\frac{d}{dt}\|S_N(D_x)\chi_j h^{n+1}_\ell\|^2+\kappa\|S_N(D_x)\chi_j
h^{n+1}_{\ell+1}\|^2 \\ \notag
&\hskip2cm =(G_1^+ -G_1^-+G_2+G_3, S_N(D_x)^2 \chi_j^2 h^{n+1}_\ell).
\end{align}
Here and in what follows,
the norm
$\| \ \|$ and inner product $( \ , \ )$ are those of
$L^2(\RR^6_{x,v})$ unless otherwise stated.
We shall evaluate the inner products on the right hand side.
Observe
that Lemma \ref{E-lemm3.111a} gives, for $ t\in[0,T]$, since  $\gamma\le 1$ is assumed, and
putting
$\tilde h^n_{\ell}=\phi_2(x-a)W_\ell\pa^\alpha_\beta g^n$,
\begin{align*}
\Big|(G_1^+, S_N^2\chi_j^2 h^{n+1}_\ell)\Big|&= \Big|(S_N\chi_jW_{-1}G_1^+,
S_N\chi_j h^{n+1}_{\ell+1})\Big|
\\&\le C \|W_{-1}G_1^+\| \,
\|S_N \chi_j h^{n+1}_{\ell+1}\|
\\&
\le C
\|h^n_{\ell-1+\gamma^+}\|\
\|\tilde h^n_{\ell-1+\gamma^+}
\|\
 \, \|S_N\chi_j h^{n+1}_{\ell+1}\|
\\
&
\le C \|g^n\|^2_{H^{k,\ell-1+\gamma^+}_{\it ul}(\RR^6)} \,
\|S_N\chi_j h^{n+1}_{\ell+1}\|
\\&
\le \frac{C}{ \kappa}\|g^n\|^4_{H^{k,\ell}_{\it ul}(\RR^6)}
+\frac{\kappa}{4}
\|S_N\chi_j h^{n+1}_{\ell+1}\|^2,
\\
\Big|(G_1^-, S_N^2\chi_j^2 h^{n+1}_\ell)\Big|&= \Big|(S_N\chi_jW_{-1}G_1^-,
S_N\chi_j h^{n+1}_{\ell+1})\Big|
\\&\le C \|W_{-1}G_1^-\| \,
\|S_N \chi_j h^{n+1}_{\ell+1}\|
\\&
 \le C \|h^n_{\ell-1+\gamma^+}\|\
\|\tilde h^{n+1}_{\ell-1+\gamma^+}
\|\
  \, \|S_N\chi_jW_{\ell+1} h^{n+1}_{\ell+1}\|
 \\&
 \le C \|g^n\|_{H^{k,\ell-1+\gamma^+}_{\it ul}(\RR^6)} \
 \|g^{n+1}\|_{H^{k,\ell-1+\gamma^+}_{\it ul}(\RR^6)}
 \|S_N\chi_jW_{\ell+1} h^{n+1}_{\ell+1}\|
 \\&
 \le \frac{C}{ \kappa}\|g^n\|^2_{H^{k,\ell}_{\it ul}(\RR^6)}
 \|g^{n+1}\|_{H^{k,\ell}_{\it ul}(\RR^6)}^2
+\frac{\kappa}{4}
 \|S_N\chi_jW_{l+1} h^{n+1}_{\ell+1}\|^2.
\end{align*}
Here  $C$ is a positive constants independent of $\kappa$,
and we have used
\[
\|h^{n}_\ell\|\le C\|g^{n}\|_{H^{k,\ell}_{\it ul}(\RR^6)}.
\]
and similarly for $\tilde h^n_\ell$.
 The estimation on the remaining two inner products are more straightforward
and can be given as follows. With some abuse of notation,
\begin{align*}
|G_2|&\le|v|W_\ell|\nabla_x\phi_1(x-a))| \ |\phi_2(x-a)\pa^\alpha_\beta g^{n+1}|
+\phi_1(x-a)W_\ell\eta_{\beta}|\pa^{\alpha+1}_{\beta-1} g^{n+1}|
\\ \intertext{where $\eta_\beta=0$ for $\beta=0$ and $=1$ for $|\beta|\ge 1$,while }
|G_3|&\le
\kappa CW_\ell\sum_{|\beta'|=1,2}\phi_1(x-a)
( |v|+1)|\partial^{\alpha}_{\beta -\beta'}
 g^{n+1}|
.
 \\ \intertext{Therefore we get}
\Big|(G_2+G_3, \, &S_N^2\chi_j^2h^{n+1}_\ell)\Big|
=\Big|W_{-1}(G_2+G_3), \, S_N^2\chi_j^2W_1 h^{n+1}_\ell)\Big|
\\&\le C  \|W_{-1}
(G_2+G_3)\| \,\|S_N\chi_j h^{n+1}_{l+1}\|
\\&
\le C(1+\kappa)\|g^{n+1}\|_{H^{k,\ell}_{\it ul}(\RR^6)}
\,\|S_N\chi_j h^{n+1}_{l+1}\|
\\&
\le \frac{C'(1+\kappa)^2}{\kappa}\|g^{n+1}\|_{H^{k,\ell}_{\it ul}(\RR^6)}^2
+\frac{\kappa}{2} \|S_N\chi_j
h^{n+1}_{\ell+1}\|^2.
\end{align*}
The constants $C'$ are independent of $\varepsilon$ and
$\kappa$.

Putting together all the estimates obtained above in \eqref{diff1}
yields
\begin{align*}
\frac{d}{dt}&\|S_N\chi_j
h^{n+1}_\ell\|^2+\kappa\|S_N\chi_j h^{n+1}_{\ell+1}\|^2
 \\&\le
\frac{C''}{ \kappa}((1+\kappa)^2+\|g^n\|^2_{H^{k,\ell}_{\it ul}(\RR^6)})
\|g^{n+1}\|_{H^{k,\ell}_{\it ul}(\RR^6)}^2 +\frac{C}{ \kappa}\|g^n\|^4_{H^k_{l}(\RR^6)}.
\end{align*}
The constants $C,C''$ are independent of $\varepsilon$ and
$\kappa$.
Integrate this over $[0,t]$ to deduce
\begin{align*}
\|S_N&\chi_j
h^{n+1}_\ell(t)\|^2+\kappa\int_0^t\|S_N\chi_j h^{n+1}_{\ell+1}(\tau)\|^2d\tau
 \\&\le\|S_N\chi_j
h^{n+1}_\ell(0)\|^2+
C_1K_n
\int_0^t\|g^{n+1}(\tau)\|_{H^{k,\ell}_{\it ul}(\RR^6)}^2d\tau
 +\frac{C_2}{ \kappa}\int_0^t\|g^n(\tau)\|^4_{H^{k,\ell}_{\it ul}(\RR^6)}d\tau.
\end{align*}
where
\[
K_n=\frac{1}{\kappa}\Big(\|g^n\|^2_{L^\infty(]0,T[;H^{k,\ell}_{\it ul}(\RR^6))}+(1+\kappa)^2 \Big),
\]
and
  $C_1>0$ is a constant independent of $\varepsilon, \kappa$
while $C_2$ is independent of $ \kappa$ but depends on
$\varepsilon$. It is easy to see that we can now take the limit
 $N \rightarrow \infty$ and $j\to \infty$, which results in
 \begin{align*}
\|
h^{n+1}_\ell(t)\|^2&+\kappa\int_0^t\|h^{n+1}_{\ell+1}(\tau)\|^2d\tau
 \\&\le\|
h^{n+1}_\ell(0)\|^2+
C_1K_n
\int_0^t\|g^{n+1}(\tau)\|_{H^{k,\ell}_{\it ul}(\RR^6)}^2d\tau
 +\frac{C_2}{ \kappa}\int_0^t\|g^n(\tau)\|^4_{H^{k,\ell}_{\it ul}(\RR^6)}d\tau.
\end{align*}
Sum up this for $|\alpha+\beta|\le k$ (see \eqref{hnl}) and take the speremum of the left hand side with respect to $a\in\RR^3$. Knowing that the right hand side is independent of $\alpha,\beta$
and also of $a\in\RR^3$,
we have
\begin{align*}
\|g^{n+1}(t)&\|_{H^{k,\ell}_{\it ul}(\RR^6)}^2
+\kappa\|g^{n+1}\|_{\cM^{k,\ell+1}(]0,t[\times\RR^6)}^2
\\&\le\|
g_0\|_{H^{k,\ell}_{\it ul}(\RR^6)}^2+
C_1K_n
\int_0^t\|g^{n+1}(\tau)\|_{H^{k,\ell}_{\it ul}(\RR^6)}^2d\tau
 +\frac{C_2}{ \kappa}\int_0^t\|g^n(\tau)\|^4_{H^{k,\ell}_{\it ul}(\RR^6)}d\tau.
\end{align*}
The Gronwall inequality then gives
\begin{align}\label{gron}
\|g^{n+1}(t)&\|_{H^{k,\ell}_{\it ul}(\RR^6)}^2
+\kappa\|g^{n+1}\|_{\cM^{k,\ell+1}(]0,t[\times\RR^6)}^2
\\&\le \notag
e^{C_1K_nt}\|
g_0\|_{H^{k,\ell}_{\it ul}(\RR^6)}^2
 +\frac{C_2}{ \kappa}\int_0^te^{C_1K_n(t-\tau)}\|g^n(\tau)\|^4_{H^{k,\ell}_{\it ul}(\RR^6)}d\tau.
\end{align}
for all $t\in[0,T]$.

Now the proof of Lemma \ref{estimate1} is completed.\end{proof}

\smallbreak We are now ready to prove the convergence
of $\{g^n\}_{n\in\NN}$. Fix ${ \kappa>0} $. Let $D_0, g_0$ be as in
Theorem \ref{E-exist-cut-off} and introduce an induction hypothesis
\begin{equation}\label{E-3.101}
\|g^{n}\|_{L^\infty(]0, T[;\,\, H^{k,\ell}_{\it ul}(\RR^6))}\leq 2D_0.
\end{equation}
for some $T\in\,]0,T_0]$. Notice that the factor $2$ can be  any
number $>1$.

\eqref{E-3.101} is true for $n=0$ due to \eqref{initial}. Suppose
that this is true for some $n>0$. We shall determine $T$ independent
of $n$. A possible choice is given by
\[
e^{C_1K_0 T}= 2, \quad \frac{2^4C_2}{\kappa}TD_0^2= 1
\quad\text{
where\quad}
K_0=\frac{1}{\kappa}(2D_0+(1+\kappa)^2)
\]
or
\begin{align}\label{Tchoice}
T=\min\left\{\frac{\log 2}{ C_1K_0}, \,
\frac{\kappa}{2^4C_2D_0^2}\right\}.
\end{align}
In fact, \eqref{gron} and \eqref{E-3.101} imply that
$g^{n+1}\in Y$ and
\begin{align*}
||g^{n+1}||_{H^{k,\ell}_{\it ul}(\RR^6)}^2&\le e^{C_1K_0 T}
\Big(\|g_0\|^2_{H^{k,\ell}_{\it ul}(\RR^6)}+\frac{C_2}{\kappa}
T||g^n||^4_{L^\infty(]0, T[;\,\, H^{k,\ell}_{\it ul}(\RR^6))}\Big)
\\&
\le e^{C_1K_0 T} \Big(D_0^2+\frac{C_2}{\kappa} T2^4D_0^4\Big)\le
4D_0^2.
\end{align*}
That is, the induction hypothesis \eqref{E-3.101} is fulfilled for
$n+1$, and hence  holds for all $n$.

 For the convergence, set $w^n=g^{n}(t)-g^{n-1}(t)$, for which
\eqref{iteration} leads to
\begin{align*}
\left\{\begin{array}{ll}
\partial_t w^{n+1}
+v\cdot\nabla_x w^{n+1}+\kappa\langle |v|
\rangle^2 w^{n+1} =\Gamma_\varepsilon^{t,+} (w^n,\,
g^n)
&\\[0.3cm]
\qquad\qquad +\Gamma_\varepsilon^{t,+} (g^{n-1},\, w^{n})
-
\Gamma_\varepsilon^{t,-} (w^n,\, g^{n+1})-\Gamma_\varepsilon^{t,-}
(g^{n-1},\, w^{n+1}),
 \\[0.3cm]
w^{n+1}|_{t=0}=0.&
\end{array}\right.
\end{align*}
By the same computation as in \eqref{diff1},
we get
\begin{align*}
||w^{n+1}||_Y^2 \leq \frac 12 C_2&e^{C_1K_0 T}\frac{1}{\kappa}T
 \Big\{\| g^{n+1}\|^2_{L^\infty(]0, T[;\,\,H^{k,\ell}_{\it ul}(\RR^6))}
 +\|g^n\|^2_{L^\infty(]0, T[;\,\,
H^{k,\ell}_{\it ul}(\RR^6))}
\\&
+\|g^{n-1}\|^2_{L^\infty(]0, T[;\,\, H^{k,\ell}_{\it ul}(\RR^6))}\Big\}
\|w^{n}\|^2_{L^\infty(]0, T[;\,\, H^{k,\ell}_{\it ul}(\RR^6))},
\end{align*}
with the same constants $C_1, C_2$ and $K_0$ as above. Then, \eqref{E-3.101} and
\eqref{Tchoice} give
\[
||g^{n+1}-g^n||_Y^2 \\
\le 2^4C_2D_0^2\kappa^{-1}T
 \|g^{n}-g^{n-1}\|^2_{L^\infty(]0, T[;\,\, H^{k,\ell}_{\it ul}(\RR^6))}.
\]
Finally,  choose  $T$ smaller if necessary so that
\[
2^4C_2D_0^2\kappa^{-1}T \leq \frac 14.
\]
Then, we have proved that for any $n\geq 1,$
\begin{equation}\label{E-3.102}
||g^{n+1}-g^n||_Y\leq \frac 12\,\, ||g^{n}-g^{n-1}||_Y.
\end{equation}
Consequently, $\{g^{n}\}$ is a convergence sequence in $Y$, and the limit
\[
g^\varepsilon\in Y,
\]
is therefore a non-negative solution of the Cauchy problem
(\ref{E-Cauchy-cut-off}). The estimate (\ref{E-3.102}) also implies
the uniqueness of solutions.

By means of the mild form \eqref{mild}, it can be proved also that for each $n$,
\[
g^{n}\in C^0([0,T]; H^{k,\ell}_{\it ul}(\RR^6))
\]
and hence so is the limit $g^\varepsilon$.
The non-negativity of $g^\varepsilon$ follows because $g^n \geq 0$.
Now the  proof of Theorem \ref{E-exist-cut-off} is completed.

\section{Uniform Estimate}\label{s4}
\setcounter{equation}{0}

We now prove the convergence of approximation sequence
$\{g^\varepsilon\}$ as $\varepsilon\rightarrow0$. The first step is to
prove the uniform boundedness of this approximation sequence. Below, the constant $C$ are various constants independent
of $\varepsilon>0$. 

\begin{theo}\label{E-uniform-estimate}
Assume that $0<s<1,\,\,  \gamma>-3/2, \ \gamma +2s< 1$. Let $g_0\in H^{k,\ell}_{\it ul}(\RR^6),
g_0\geq 0$ for some $k\geq 4,\,\, l\geq 3$. Then there exists $T_* \in ]0,T_0]$
depending on $\|g_0\|_{H^{k,\ell}_{\it ul}(\RR^6)}$ but not on $\varepsilon$ such that
  if for some $0<T\leq T_0$\, ,
\begin{equation}\label{assume}
g^\varepsilon\in C^0(]0, T];\,\, H^{k,\ell}_{\it ul}(\RR^6))\cap
\cM^{k, \ell+1}(]0,T[\times\RR^6)
\end{equation}
is a non-negative solution of the Cauchy problem
(\ref{E-Cauchy-cut-off}) and if
$T_{**}=\min\{T,\,\,T_*\}$,
then it holds that
 \begin{equation}\label{E-3.103}
\|g^\varepsilon\|_{L^\infty(]0, T_{**}[;\,\, H^{k,\ell}_{\it ul}(\RR^6))}\leq 2
\|g_0\|_{H^{k,\ell}_{\it ul}(\RR^6)}.
\end{equation}
 \end{theo}
\begin{rema}
The case $T_*\le T$ gives a uniform estimate
of local solutions on
the fixed time interval $[0,T_*]$
while the case $T<T_*$ gives an a priori estimate on the
existence time interval $[0,T]$ of local solutions. The latter is used for the continuation argument of
local solutions, in Subsection 4.4 below.
\end{rema}
In the following, $\rho>0,\,\kappa>0$ are fixed. Furthermore, recall
$T_0=\rho/(2\kappa)$. We start with  a solution $g^\varepsilon$ subject
to \eqref{assume} for some  $T\in\,]0,T_0]$. Put
\begin{align}\label{hepsl}
h^{\alpha,\beta}_\ell
=\phi_1(x-a)W_\ell\partial^\alpha_\beta g^\varepsilon
\end{align}
and  take the $L^2$ inner product of  it and the equation for it. As before
$\|\ \|$ and $(\ ,\ )$ stand for the $L^2(\RR^6)$ norm and inner product
respectively unless otherwise stated.
 Then we have
\begin{align}\label{E-equ-2}
\frac 12 \frac{d}{dt}&\|h^{\alpha,\beta}_\ell\|^2
 + \kappa
\| h^{\alpha,\beta}_{\ell+1}\|^2
=(\Xi, h^{\alpha,\beta}_\ell),
\end{align}
where
\begin{align*}
\Xi&=\phi_1(x-a)W_\ell\partial^\alpha_\beta \Gamma(g^\varepsilon,g^\varepsilon)
-[ \phi_1(x-a)W_\ell\partial^\alpha_\beta, v\cdot\nabla_x]g^\varepsilon
-\kappa[ \phi_1(x-a)W_\ell\partial^\alpha_\beta, \langle v\rangle^2]g^\varepsilon
\\&=\Xi_1+\Xi_2+\Xi_3.
\end{align*}
We shall derive the estimates
\begin{lemm}\label{lem410}
Assume that $ 0<s<1/2,  \gamma\ge -3/2,   \gamma+2s< 1$. Then,
\begin{align*}
&|(\Xi_1,h^{\alpha,\beta}_\ell)|\le C \|g^\varepsilon_{\ell}\|_{H^{k,\ell}_{\it ul}}^2
\sum_{|\alpha'+ \beta'|\le k}\|h^{\alpha', \beta'}_{\ell+1}\|,
\\&
|(\Xi_2+\Xi_3,h^{\alpha,\beta}_\ell)|\le C
 (1+\kappa+\|g^\varepsilon\|_{H^{k,\ell}_{\it ul}(\RR^6)})
\|g^\varepsilon\|_{H^{k,\ell}_{\it ul}(\RR^6)}
\|h^{\alpha,\beta}_{\ell+1}\|.
\end{align*}
\end{lemm}
\noindent This lemma will be proved in Section 4. Now we have
\[
|(\Xi,h^{\alpha,\beta}_\ell)|\le C
((1+\kappa)^2+\|g^\varepsilon\|_{H^{k,\ell}_{\it ul}(\RR^6)}^2)\|g^\varepsilon\|_{H^{k,\ell}_{\it ul}(\RR^6)}^2+\frac{\kappa}{k^3}\sum_{|\alpha'+ \beta'|\le k}\|h^{\alpha', \beta'}_{\ell+1}\|^2
\]
whence \eqref{E-equ-2} yields
\begin{align*}
 \frac{d}{dt}&
\|h^{\alpha,\beta}_\ell(t)\|^2
 + 2\kappa 
\| h^{\alpha,\beta}_{\ell+1}\|^2
\le  \frac{C}{\kappa}((1+\kappa)^2+\|g^\varepsilon\|_{H^{k,\ell}_{\it ul}(\RR^6)}^2)\|g^\varepsilon\|_{H^{k,\ell}_{\it ul}(\RR^6)}^2+
\frac{\kappa}{k^3}\sum_{|\alpha'+ \beta'|\le k}\|h^{\alpha', \beta'}_{\ell+1}\|^2
,
\end{align*}
and after integrating over $]0,t[$,
\begin{align*}
\|h^{\alpha,\beta}_\ell&(t)\|^2
 + \kappa 
\int_0^t
\| h^{\alpha,\beta}_{\ell+1}(\tau)\|^2d\tau
\\&
\le  \|g^{\alpha,\beta}(0)\|_{H^{k,\ell}_{\it ul}(\RR^6)}^2+\frac{C}{\kappa}\int_0^t((1+\kappa)^2+\|g^\varepsilon(\tau)\|_{H^{k,\ell}_{\it ul}(\RR^6)}^2)\|g^\varepsilon(\tau)\|_{H^{k,\ell}_{\it ul}(\RR^6)}^2d\tau
\\&\quad +\frac{\kappa}{k^3}\sum_{|\alpha'+ \beta'|\le k}\int_0^t\|h^{\alpha', \beta'}_{\ell+1}(\tau)\|^2d\tau.
\end{align*}
Take the spremum with respect to $a\in\RR^3$ (see \eqref{hepsl}) and sum up over $|\alpha+\beta|\le k$
to deduce that
\begin{align*}
 \|g^\varepsilon(t)&\|_{H^{k,\ell}_{\it ul}(\RR^6)}^2
 + \kappa\|g^\varepsilon\|_{\cM^{k,\ell+1}(]0,t[\times\RR^6)}^2
\\&\le   \|g_0\|_{H^{k,\ell}_{\it ul}(\RR^6)}^2
+\frac{C}{\kappa}\int_0^t(1+\|g^\varepsilon(\tau)\|_{H^{k,\ell}_{\it ul}(\RR^6)})^2
\|g^\varepsilon(\tau)\|_{H^{k,\ell}_{\it ul}(\RR^6)}^2d\tau.\notag
\end{align*}
Then the Gronwall type inequality gives for $C_\kappa=C/\kappa$,
$$
\|g^\varepsilon(t) \|^2_{H^{k,\ell}_{\it ul}(\RR^6)}\leq
\frac{\|g_0\|^2_{H^{k,\ell}_{\it ul}(\RR^6)}e^{{C_\kappa} t}} {1-\big(e^{{C_\kappa}
t}-1\big)\|g_0\|^2_{H^{k,\ell}_{\it ul}(\RR^6)}}\,, \enskip
$$
as long as the denominator remains  positive.
We choose $T_*>0$ small enough such that
$$
\frac{e^{{C_\kappa} T_*}} {1-\big(e^{{C_\kappa}
T_*}-1\big)\|g_0\|^2_{H^{k,\ell}_{\it ul}(\RR^6)}}=4.
$$
Then
$$
T_*=\frac{1}{C_\kappa}\log\Big(1+\frac{3}{1+4\|g_0\|^2_{H^{k,\ell}_{\it ul}(\RR^6)}}\Big)
$$
is independent of $\varepsilon>0$, but depends on
$\|g_0\|_{H^{k,\ell}_{\it ul}(\RR^6)}$ and the constant $C$ which depends on
$\rho, \kappa, k$ and $l$. Now  we
have  (\ref{E-3.103}) for $T_{**} = \min(T,T_*)$.

From (\ref{E-3.103}) and (\ref{E-equ-2}), we get also, for
$\kappa>0$,
$$
\kappa  \|g^\varepsilon\|^2_{\cM^{k,\ell+1}(]0, T_{**}[\times\RR^6})
\leq 2 \|g_0\|^2_{H^{k,\ell}_{\it ul}(\RR^6)} \Big ( 1+ 2C T_{*} (1 + 2
\|g_0\|^2_{H^{k,\ell}_{\it ul}(\RR^6)}) \Big)  .
$$
We have proved Theorem \ref{E-uniform-estimate}.

Combine  Theorems \ref{E-exist-cut-off} and \ref{E-uniform-estimate} and use the compactness argument as
in  Section 4.4 of \cite{amuxy-arma}, to conclude the existence
part of Theorem \ref{E-theo-0.2}. The uniqueness part comes from Theorem 4.1 of
\cite{amuxy4-4}.
Now the main theorem \ref{E-theo1} is proved by the help of Theorem \ref{E-theo-0.2},
in the same manner 
as  in Section 4.5 of \cite{amuxy-arma}.

\section{Proof of Lemma \ref{lem410}}
In the sequel, the notation $A\lesssim B$ means that there is a 
constant $C$ independent of $A,B$ such that $A\le CB$, and similarly
for $A\gtrsim B$.
We start with
\subsection{Esimate of $\Xi_1$}
Notice that
\begin{align*}
\Xi_1&=\phi_1(x-a)W_\ell\pa^\alpha_\beta\Gamma(g^\varepsilon,g^\varepsilon)
\\&=\sum_{\begin{subarray}{l}\alpha_1+\alpha_2=\alpha
\\
 \beta_1+\beta_2+\beta_3=\beta
\end{subarray}} C^{\alpha_1, \alpha_2}_{\beta_1,
 \beta_2,
 \beta_3}\notag
 W_\ell{\cT}_\varepsilon(\phi_2(x-a)\partial^{\alpha_2}_{\beta_2}
 g^\varepsilon,\,\,\phi_1(x-a)\partial^{\alpha_1}_{\beta_1} g^\varepsilon,\,\,
 \partial_{\beta_3}\mu(t)).
 \end{align*}
 Put
 \begin{align*}
 F=\phi_2(x-a)\partial^{\alpha_2}_{\beta_2}
 g^\varepsilon,
\quad G=\phi_1(x-a)\partial^{\alpha_1}_{\beta_1}
 g^\varepsilon, \quad  M=\pa_{\beta_3}\mu,
\end{align*}
and write
 \begin{align*}
 W_\ell{\cT}_\varepsilon(F,G,M)
&=W_\ell Q(MF, G)
 + W_\ell\iint_{\RR^3\times\SS^2}B(M_*-M'_*)F_*' G'dv_*d\sigma
 \\&=A_1+A_2.
\end{align*}
{\bf Estimate of $A_1$:}

(1) The case $\alpha_1=\alpha, \beta_1=\beta$: Notice that then $\alpha_2=\beta_2=\beta_3=0$.
Write
\[
A_1=Q(MF, W_\ell G)+[W_{\ell} Q(MF, G)-Q(MF,W_\ell G)]=A_{10}+A_{11}.
\]
Here,  $F=\phi_2(x-a)g^\varepsilon$, $W_{\ell}G=h^{\alpha,\beta}_\ell$, and $M=\mu$. We start with
\begin{lemm}
It holds that
\begin{align*}
(A_{10}, &h^{\alpha,\beta}_\ell)=-\frac 12 D+ A_{101}
\end{align*}
where
\begin{align}\label{D}
D=\iiiint_{\RR^9\times\SS^2}B(\mu F)_*
 ((h')^{\alpha,\beta}_\ell-h^{\alpha,\beta}_\ell)^2dxdv dv_*d\sigma
 \ge 0
\end{align}
while  $A_{101}$ enjoys the estimate
that for $0<s<1$ and $-3/2<\gamma\le 1$, and  for any $k\ge 2$ and $\ell\ge 1$,
\begin{align*}
|A_{101}|\lesssim \|g^\varepsilon\|^2_{H^{k,l}_{\it ul}(\RR^6)}\|h^{\alpha,\beta}_{\ell+1}\|.
\end{align*}
\end{lemm}
\begin{proof}
Put $H=h^{\alpha,\beta}_\ell$. Then,
\begin{align*}
(A_{10}, &h^{\alpha,\beta}_\ell)=(Q(\mu F, W_\ell G),H)=(Q(\mu F, H),H)
\\&=
\iiiint_{\RR^9\times\SS^2}B(\mu F)_*H
(H'-H)dxdv dv_*d\sigma
\\&\quad =
-\frac 12 \iiiint_{\RR^9\times\SS^2}B(\mu F)_*
(H'-H)^2dxdv dv_*d\sigma
 \\&\quad\quad +\frac 12 \iiiint_{\RR^9\times\SS^2}B(\mu F)_*[(H')^2-(H)^2] dxdvdv_*d\sigma
\\&=-\frac 12 D+A_{101}.
\end{align*}

Thanks to the cancellation lemma in \cite{al-1}
 we get with $S(z)\lesssim |z|^\gamma$ for the inverse power potential ,
 \begin{align*}
 |A_{101}|&\lesssim  \iint_{\RR^6}|
\mu F(v_*)| \ |S(v_*)*_{v}H^2|
dv_*dx.
\end{align*}
Since
 \begin{align}\label{3.8}
\int_{\RR^3} |v-v_*|^{\gamma}(\mu F)_*dv_*\lesssim \langle v\rangle^{\gamma}
 \| F\|_{L^2_v}
 \end{align}
holds true
for $\gamma>-3/2$,
 we now have
 \begin{align*}
|A_{101}|\lesssim
 \int_{\RR^3} \|F\|_{L^2_v}\|H\|_{L^2_v}
 \|H \|_{L^2_{\gamma,v}}dx
\lesssim \|F\|_{H^2_x(L^2_v)}\|H\|\
 \|W_\gamma H \|
 \end{align*}
 Hence since  $\gamma\le 1$ is assumed,  we get
 \begin{align*}
|A_{101}|
\lesssim \| \phi_2(x-a)g^\varepsilon\|_{H^2(\RR^3_x;L^2(\RR^3_v)}
\| h^{\alpha,\beta}_\ell\|\ \|h^{\alpha,\beta}_{\ell+1}\|
\lesssim \| g^\varepsilon\|_{H^{k,\ell}_{\it ul}(\RR^6)}^2
\|h^{\alpha,\beta}_{\ell+1}\|,
\end{align*}
which proves the lemma.
\end{proof}

The estimate of $A_{11}$ is stated as follows.
  \begin{lemm}
Let $0<s <1$ and $-5/2<\gamma\le 1$. Then,  for $k\ge 0, \ell\ge 2$,
\begin{align*}
\Big|(A_{11}, h^{\alpha,\beta}_{\ell})_{L^2(\RR^6)}\Big|
&
\le\frac 14 D+C
\|g^\epsilon\|_{H^{k,\ell}_{\it ul}(\RR^6)}^2\|h^{\alpha,\beta}_{\ell+1}\|,
\end{align*}
where $D$ is defined by \eqref{D}.
\end{lemm}

\begin{proof}
With $W_\ell G=h^{\alpha,\beta}_\ell=H$,
\begin{align*}
\big(W_{\ell} & \,\, Q(MF,\,\,G)-Q(MF,\,\,W_\ell  \,\,
G),\,\,\, H\big)_{L^2(\RR^3_v)}
\\&=
 \iiint_{\RR^6\times \SS^2}B(W_\ell'-W_\ell)
 (\mu F)_*G H'dvdv_*d\sigma
 \\&=\iiint_{\RR^6\times \SS^2}B(W_\ell'-W_\ell)
 (\mu F)_*G Hdvdv_*d\sigma
 \\&\quad+\iiint_{\RR^6\times \SS^2}B(W_\ell'-W_\ell)
 (\mu F)_*G (H'-H)dvdv_*d\sigma
 \\&=K_1+K_2.
 \end{align*}
 For $K_1$, we use the Taylor formula of second order
\[
W_\ell'-W_\ell = \nabla \Big( W_\ell\Big)\cdot (v'-v)
+ \frac{1}{2} \int_0^1 \nabla^2 \Big( W_\ell(v_\tau ) \Big)d\tau(v'-v)^{\otimes 2 }
\]
and write $K_1 = K_{1,1} + K_{1,2}$ in terms of this decomposition. Note that
\[ v' -v = \frac{|v-v_*|}{2} (\sigma-(k\cdot \sigma) k)  + \frac{|v-v_*|}{2}
( k\cdot \sigma-1)k,
\]
where
$
k=(v-v_*)/|v-v_*|.
$
Then, since  it follows from the symmetry that the integral corresponding to the first term vanishes,
we have, using \eqref{3.8},
\begin{align*}
|K_{1,1}|&=\Big|\iiint b(\cos\theta)|v-v_*|^{\gamma+1}(1-\cos\theta)\Big(\nabla(W_\ell)\cdot k\Big)
(\mu F)_*G Hdvdv_*d\sigma\Big|
\\&\lesssim
\int \Big\{\int |v-v_*|^{\gamma+1}(\mu F)_*dv_*\Big\}|W_{\ell-1}G H|dv
\\&\lesssim
\|F\|_{L^2_v}\int W_{\ell+\gamma}|G H|dv
\\&\lesssim
\|F\|_{L^2_v}\|W_\ell G\|_{L^2_v}\|W_\gamma H\|_{L^2_v}.
\end{align*}

Next,  let $\ell\ge 2$. Then,  since the energy conservation property
$|v|^2+|v_*|^2=|v'|^2+|v'_*|^2$ implies
\[
|(\nabla^2 W_\ell)(v_\tau)| \lesssim W_{\ell-2}'+W_{\ell-2}
\lesssim (W_{\ell-2})_*+W_{\ell-2}\lesssim (W_{\ell-2})_*W_{\ell-2},
\]
and again using \eqref{3.8}, we get
\begin{align*}
|K_{1,2}|&\lesssim
\Big|\iiint b(\cos\theta)\theta^2|v-v_*|^{\gamma+2}((W_{\ell-2})_*W_{\ell-2})
(\mu F)_*G Hdvdv_*d\sigma\Big|
\\&\lesssim
\int \Big\{\int |v-v_*|^{\gamma+2}(W_{\ell-2}\mu F)_*dv_*\Big\}|W_{\ell-2}G H|dv
\\&\lesssim
\|F\|_{L^2_v}\int W_{\ell+\gamma}|G H|dv
\\&\lesssim
\|F\|_{L^2_v}\|W_\ell G\|_{L^2_v}\|W_\gamma H\|_{L^2_v}.
\end{align*}
Thus, by a similar computation of $A_{101}$, we conclude
\begin{align*}
\int |K_{1}|dx&\le\int |K_{11}|dx+\int |K_{12}|dx
\\&\lesssim
\|F\|_{H^2_x(L^2_v)}\|W_\ell G\|\ \|W_\gamma H\|
\lesssim \| g^\varepsilon\|_{H^{k,\ell}_{\it ul}(\RR^6)}^2
\|h^{\alpha,\beta}_{\ell+1}\|.
\end{align*}

We shall estimate $K_2$. By
 the Cauchy-Schwarz inequality and again by \eqref{3.8},
 \begin{align*}
 (\int|K_{2}|&dx)^2\le  \iiint_{\RR^6\times \SS^2}B
 (\mu F)_*| H'-H|^2dvdv_*d\sigma dx
 \\&\times
  \iiint_{\RR^6\times \SS^2}B(W_\ell'-W_\ell)^2
 (\mu F)_*G^2 dvdv_*d\sigma dx
 \\&\lesssim D\Big(\iiint_{\RR^6\times \SS^2}b(\cos\theta)\theta^2|v-v_*|^{\gamma+2}
 (W_{\ell-1}^2\mu F)_*(W_{\ell-1}G)^2 dvdv_*d\sigma dx\Big)
  \\&\lesssim D\int\Big(\int_{\RR^3}|v-v_*|^{\gamma+2}
 (W_{\ell-1}^2\mu F)_*dv_*\Big)(W_{\ell-1}G)^2 dv dx\Big)
 \\&\lesssim
D \|F\|_{H^2_x(L^2_v)}\|W_\ell G\|\ \|W_{\ell+\gamma}G\|
\\&\lesssim D\| g^\varepsilon\|_{H^{k,\ell}_{\it ul}(\RR^6)}^2
\|h^{\alpha,\beta}_{\ell+1}\|
 \end{align*}
 where $D$ is as in \eqref{D}. We now obtain
 \begin{align*}
 \int|K_{2}|dx\le \frac 14 D+ C\| g^\varepsilon\|_{H^{k,\ell}_{\it ul}(\RR^6)}^2
\|h^{\alpha,\beta}_{\ell+1}\|.
 \end{align*}
 This ends the proof of the lemma.
 \end{proof}

\noindent
{\bf Estimate of $A_1$ continued:
(2) The case $|\alpha_1+ \beta_1|\le k-1$.}
We shall establish
\begin{lemm}\label{lem3.2}
Let $0<s<1/2$ and $\gamma>\max\{-3, -2s-3/2\}$. Then, for $k\ge 4, \ell\ge 0$ and
for $|\alpha_1+ \beta_1|\le k-1$,
\begin{align*}
|(A_1, h^{\alpha,\beta}_\ell)|
\lesssim
\|g^\varepsilon\|_{H^{k,\ell}_{\it ul}}^2\|h^{\alpha,\beta}_{\ell+1}\|.
 \end{align*}
 \end{lemm}
The proof is based on the following upper bound estimate of $Q$ established in Proposition 2.9 of \cite{amuxy4-4}.
\begin{prop}\label{upperQ}
Let $0<s<1$ and $\gamma>\max\{-3, -2s-3/2\}$. Then for any $\ell\in \RR$ and $m\in[s-1, s]$,
\begin{align*}
\Big|\Big(Q(f,g) , h \Big)_{L^2_v}\Big|
\lesssim
\Big( \|f\|_{L^1_{\ell^++ (\gamma+2s)^+}}
  +
\|f \|_{L^2}\Big)\|g\|_{H^{\max\{s+m, (2s-1+\epsilon)^+\}}_{\ell^++ (\gamma+2s)^+}} \|h\|_{H^{s-m}_{-\ell}},
\end{align*}
for any $\epsilon>0$.
\end{prop}
\noindent Admit this and put
\[
 F=\phi_2(x-a)\partial^{\alpha_2}_{\beta_2}
 g^\varepsilon,
\quad G=\phi_1(x-a)\partial^{\alpha_1}_{\beta_1}
 g^\varepsilon, \quad  M=\pa_{\beta_3}\mu.
\]
Recall the norm and inner product of $L^2(\RR^6)$ are denoted by $\| \ \|$ and $(\ ,\ )$
respectively. Use the above theorem for $m=s$ and $\alpha=\ell-1$ to deduce
\begin{align*}
&|(A_{1},h^{\alpha,\beta}_\ell)|=|W_\ell Q(MF,G),h^{\alpha,\beta}_\ell)|
\\&\lesssim
\int_{\RR^3}
\Big( \|MF\|_{L^1_{\ell-1+ (\gamma+2s)^+}}
  +
\|MF \|_{L^2_{\ell-2}}\Big)\|G\|_{H^{2s}_{\ell-1+  (\gamma+2s)^+}}
\| W_{\ell}h^{\alpha,\beta}_{\ell}\|_{L^2_{-\ell+1}} dx
\\&\lesssim
\int_{\RR^3_v}\|F\|_{L^2_v}
\| G\|_{H^{2s}_{\ell}(\RR^3_v)}
 \|h^{\alpha,\beta}_{\ell+1}\|_{L^2_v}dx
\end{align*}
where $ \gamma+2s< 1$ is assumed.

Suppose $|\alpha_2+\beta_2|\le 2$. Then,
\begin{align*}
|(A_{1},h^{\alpha,\beta}_\ell)|&\lesssim
\|\phi_2(x-a)\partial^{\alpha_2}_{\beta_2}
   g^\varepsilon\|_{H^2_x(L^2_v)}\|\phi_1(x-a)W_{\ell}\partial^{\alpha_1}_{\beta_1} g^\varepsilon\|_{L^2_x(H^{2s}_v)}\
  \|h^{\alpha,\beta}_{\ell+1}\|
 \\&\lesssim
 \|g^\varepsilon\|_{H^{k,\ell}_{\it ul}}^2\|h^{\alpha,\beta}_{\ell+1}\|,
 \end{align*}
 where  $|\alpha_1+\beta_1|+2s\le |\alpha_1+\beta_1|+1\le k$  was taken into account.

On the other hand let $|\alpha_2+\beta_2|>2$.
Then $|\alpha_1+\beta_1|\le k-1-|\alpha_2+\beta_2|\le k-4$ holds and
\begin{align*}
|(A_{1},h^{\alpha,\beta}_\ell)|&\lesssim
\|\phi_2(x-a)\partial^{\alpha_2}_{\beta_2}
  g^\varepsilon\|\ \|\phi_1(x-a)W_{\ell}\partial^{\alpha_1}_{\beta_1} g^\varepsilon\|_{H^2_x(H^{2s}_{v})}
 \|h^{\alpha,\beta}_{\ell+1}\|
\\&\lesssim
\|g^\varepsilon\|_{H^{k,\ell}_{\it ul}}^2\|h^{\alpha,\beta}_{\ell+1}\|.
 \end{align*}
 Now the proof of Lemma \ref{lem3.2} is complete.

\medskip
{\bf Estimate of $A_2$.}
We shall prove
\begin{lemm}\label{lem3.4}
Let $0<s<1/2$, $\gamma>-3/2$, and $2s+\gamma<1$. Then for $k>3, \ell>5/2$
\begin{align*}
|(A_2,h^{\alpha,\beta}_{\ell})|
\lesssim
\|g^\varepsilon\|_{H^{k,\ell}_{\it ul}(\RR^6)}.
\end{align*}
\end{lemm}
\begin{proof}
Firstly, we write
\begin{align*}
A_2=&
W_\ell\iint_{\RR^3\times\SS^2}B(M_*-M'_*)F_*' G'dv_*d\sigma
\\&=\iint_{\RR^3\times\SS^2}B(W_{\ell}-W'_{\ell})(M_*-M'_*)F_*' G'dv_*d\sigma
\\&\quad +\iint_{\RR^3\times\SS^2}B(M_*-M'_*)F_*' (W_{\ell}G)'dv_*d\sigma
\end{align*}
and hence, putting $H=h^{\alpha,\beta}_{\ell}$,
\begin{align*}
|(A_2, h^{\alpha,\beta}_{\ell})|
&\le\iiiint_{\RR^9\times\SS^2}B|W_{\ell}-W_{\ell}'|\
|M_*-M'_*|\ |F_*'|\ |G|'|\ |H|dvdv_*d\sigma dx
\\&\quad +\iiiint_{\RR^9\times\SS^2}B|M_*-M'_*||F_*'||(W_{\ell}G)'|\ |H |dvdv_*d\sigma dx
\\&=A_{21}+A_{22}.
\end{align*}

If $\ell\ge 1$,
\begin{align*}
|W_{\ell}-W'_{\ell}|&\lesssim|v-v'||W_{\ell-1}+W_{\ell-1}'|
\lesssim|v'-v_*'|\theta|(W_{\ell-1})_*'+W_{\ell-1}'|
 \\&
\lesssim\theta|(W_{\ell})_*'+W_{\ell}'|\lesssim \theta(W_{\ell})_*'W_{\ell}',
\end{align*}
and knowing that $|M|\lesssim \mu(v)^{1/2}$, we get
\begin{align*}
 A_{21}&\le \iiiint_{\RR^9\times\SS^2}b(\cos\theta)
\theta|v-v_*|^{\gamma}\ \mu_*^{1/2}|(W_{\ell}F)_*' |\ |(W_{\ell}G)'|
 |H|dvdv_*d\sigma dx
\\&+
\iiiint_{\RR^9\times\SS^2}b(\cos\theta)
\theta|v-v_*|^{\gamma}\ |(\mu^{1/2}W_{\ell}F)_*' ||(W_{\ell}G)'|
 |H|dvdv_*d\sigma dx\\&= A_{211}+A_{212}.
\end{align*}

By the Schwarz inequality
\begin{align*}
 A_{211}^2\lesssim &
\iiiint_{\RR^9\times\SS^2}\theta^{-1-2s}|v-v_*|^{2\gamma}\mu_*|H|^2dvdv_*d\sigma dx
\\&\times \iiiint_{\RR^9\times\SS^2}\theta^{-1-2s}\ |(W_{\ell}F)_*' |^2\ |(W_{\ell}G)'|^2dvdv_*d\sigma dx
\\&=J_{1}J_{2}.
\end{align*}
Using \eqref{3.8}yields

\begin{align*}
J_{1}&\lesssim \iint \langle v\rangle^{2\gamma}|H|^2dvdx
\lesssim \|W_\gamma H\|^2= \|h^{\alpha,\beta}_{\ell+\gamma}\|^2\le
\|h^{\alpha,\beta}_{\ell+1}\|^2,
\\
\intertext{while by the change of variables $(v,v_*)\to (v',v_*')$}
J_{2}&\lesssim
\iint|(W_{\ell}F)_*' |^2\ |(W_{\ell}G)'|^2dvdv_*dx
\\&=\int \|W_{\ell}F\|^2_{L^2_v}
\ \|W_{\ell}G\|^2_{L^2_v}dx
\\&=\int \|W_{\ell}\phi_2(x-a)\partial^{\alpha_2}_{\beta_2}
 g^\varepsilon\|^2_{L^2_v}
\ \|h^{\alpha,\beta}_{\ell}\|_{L^2_v}^2dx
.
\end{align*}
Then, if $|\alpha_1+\beta_1|\le 2$,
\begin{align*}
J_{2}\lesssim
\|W_{\ell}\phi_2(x-a)\partial^{\alpha_2}_{\beta_2}
 g^\varepsilon\|^2
\ \|h^{\alpha,\beta}_{\ell}\|_{H^2_x(L^2_v)}^2
\lesssim \|g^\varepsilon\|_{H^{k,\ell}_{\it ul}(\RR^3)}^4,
\end{align*}
while if $|\alpha_1+\beta_1|> 2$ so that $|\alpha_2+\beta_2|\le k-|\alpha_1+\beta_1|\le k-2$,
\begin{align*}
J_{2}\lesssim
\|W_{\ell}\phi_2(x-a)\partial^{\alpha_2}_{\beta_2}
 g^\varepsilon\|_{H^2_x(L^2_v)}^2
\ \|h^{\alpha,\beta}_{\ell}\|^2
\lesssim \|g^\varepsilon\|_{H^{k,\ell}_{\it ul}(\RR^3)}^4
.
\end{align*}
In conclusion, we obtained
\begin{align*}
 A_{211}^2\lesssim &\|g^\varepsilon\|_{H^{k,\ell}_{\it ul}(\RR^6)}^2\|h^{\alpha,\beta}_{\ell+1}\|.
 \end{align*}

 We turn to $A_{212}$.
Notice
by the change of variables  $(v,v_*)\to (v',v_*')$ and by the 
Cauchy-Schwarz inequality and \eqref{3.8}
that for $\gamma>-3/2$,
\begin{align*}
A_{212}\lesssim&
\iiiint_{\RR^9\times\SS^2}\theta^{-1-2s}
|v'-v_*'|^{\gamma}\ |(\mu^{1/2}W_{\ell}F)_*' ||(W_{\ell}G)'|
 |H|dvdv_*d\sigma dx
 \\&=\iiiint_{\RR^6\times\SS^2}\theta^{-1-2s}
|v-v_*|^{\gamma}\ |(\mu^{1/2}W_{\ell}F)_* |
\ |(W_{\ell}G)|
 |H'|dvdv_*d\sigma dx
 \\&=\iiint_{\RR^9\times\SS^2}\theta^{-1-2s}\Big(\int
|v-v_*|^{\gamma}\ |(\mu^{1/2}W_{\ell}F)_* |dv_*\Big)
\ |(W_{\ell}G)|
 |H'|dvd\sigma dx
\\&\lesssim
\int \| F\|_{L^2_v}
\Big\{\iint_{\RR^3\times\SS^2}\theta^{-1-2s}\langle v\rangle^{\gamma}
|(W_{\ell}G)|
 |H'|dvd\sigma\Big\} dx.
 \end{align*}
By the regular change of variable $v\to v'$,
\begin{align*}
A_{212}&\lesssim \int \| \phi_2(x-a)\partial^{\alpha_2}_{\beta_2}
 g^\varepsilon\|_{L^2_v}\|h^{\alpha_1, \beta_1}_{\ell+\gamma}\|_{L^2_v}
\|h^{\alpha, \beta}_{\ell}\|_{L^2_v}dx,
\end{align*}
the last integral of which is bounded, by the Sobolev embedding,  by
\begin{align*}
 \| \phi_2(x-a)\partial^{\alpha_2}_{\beta_2}
 g^\varepsilon\|_{H^2_x(L^2_v)}\|h^{\alpha_1, \beta_1}_{\ell+\gamma}\|\
\|h^{\alpha, \beta}_{\ell}\|
\lesssim \|g^\varepsilon\|^2_{H^{k,\ell}_{\it ul}(\RR^6)}\|h^{\alpha_1, \beta_1}_{\ell+\gamma}\|
\end{align*}
when $|\alpha_2+\beta_2|\le 2$, and by
\begin{align*}
 \| \phi_2(x-a)\partial^{\alpha_2}_{\beta_2}&
 g^\varepsilon\|\ \|h^{\alpha_1, \beta_1}_{\ell+\gamma}\|_{H^2_x(L^2_v})\
\|h^{\alpha, \beta}_{\ell}\|
\lesssim \|g^\varepsilon\|^2_{H^{k,\ell}_{\it ul}(\RR^6)}
\sum_{|\alpha'+ \beta'|\le k}\|h^{\alpha', \beta'}_{\ell+\gamma}\|
\end{align*}
when $|\alpha_2+\beta_2|> 2$ for which $|\alpha_1+\beta_1|\le k-|\alpha_2+\beta_2|\le
k-2
$.
Thus we obtained
\begin{align*}
|A_{21}|\lesssim \|g^\varepsilon\|^2_{H^{k,\ell}_{\it ul}(\RR^6)}
\sum_{|\alpha'+ \beta'|\le k}\|h^{\alpha', \beta'}_{\ell+\gamma}\|.
\end{align*}

In remains to evaluate $A_{22}$. First,  an interpolation of the  Taylor formula
and the boundedness $|M|\le C$ yields
\begin{align*}
|M_*-M'_*| \le C|v_*-v_*'|^\lambda\le C\theta^\lambda
|v-v_*|^\lambda \le C\theta^\lambda |v'-v'_*|^\lambda.\quad
\end{align*}
Since $s\in(0,1/2)$ and $\gamma+2s<1$ are assumed, there is $\lambda\in(0,1)$ such that
 $\lambda>2s, \gamma+\lambda< 1$. Therefore
after the change of variable $(v,v_*)\to (v',v_*)$,
\begin{align*}
A_{22}&= \iiiint_{\RR^9\times\SS^2}B|M_*-M'_*||F_*'||(W_{\ell}G)'|\ |H |dvdv_*d\sigma dx
\\&\lesssim
 \iiiint_{\RR^9\times\SS^2}\theta^{-2-2s+\lambda}|v'-v'_*|^{\gamma+\lambda}
|F_*'||(W_{\ell} G)'|
|H|dvdv_*d\sigma dx
\\&=
\iiiint_{\RR^9\times\SS^2}\theta^{-2-2s+\lambda}|v-v_*|^{\gamma+\lambda}
|F_*||(W_{\ell} G)|
|H'|dvdv_*d\sigma dx
\end{align*}
We shall check the two cases.

(1) The case $\gamma+\lambda\ge 0$:
Then, noting $|v-v_*|^{\gamma+\lambda}\lesssim \langle v\rangle^{\gamma+\lambda}
\langle v_*\rangle^{\gamma+\lambda}$ and
using the regular change of variable $v\to v'$,
\begin{align*}
A_{22}&\lesssim
\iiiint_{\RR^9\times\SS^2}\theta^{-2-2s+\lambda}
|(W_{\gamma+\lambda} F)_*|\
|W_{\ell+\gamma+\lambda} G|
|H'|dvdv_*d\sigma dx.
 \\&\lesssim\int \|W_1 F\|_{L^1_v}\Big\{\iint_{\RR^3\times\SS^2}\theta^{-2-2s+\lambda}
\ |W_{\ell+1} G|
 |H'|dvd\sigma\Big\} dx
 \\&\lesssim\int
\|W_{\ell_0} F\|_{L^2_v}\|W_{\ell+1} G\|_{L^2_v}
 \|H\|_{L^2_v} dx
 \\&=\int
\|W_{\ell_0}\phi_2(x-a)\partial^{\alpha_2}_{\beta_2}
 g^\varepsilon\|_{L^2_v}
\|h^{\alpha_1,\beta_1}_{\ell+1}\|_{L^2_v}\|h^{\alpha,\beta}_{\ell}\|_{L^2_v}dx
 \end{align*}
for $\ell_0>1+3/2=5/2$.
By the Sobolev embedding applied separately as before to the 
cases $|\alpha_1+\beta_1|\le 2$ and $|\alpha_1+\beta_1|>  2$,
We finally obtain
\[
A_{22}\lesssim \|g^\varepsilon_{\ell}\|_{H^{k,\ell}_{\it ul}}^2
\sum_{|\alpha'+ \beta'|\le k}\|h^{\alpha', \beta'}_{\ell+1}\|
.
\]

(2) The case $\gamma+\lambda< 0$:
We shall  use the following split of integral,
\begin{align*}
A_{22}&\lesssim
\iiiint_{\RR^9\times\SS^2}\theta^{-2-2s+\lambda}|v-v_*|^{\gamma+\lambda}
|F_*||(W_{\ell} G)|
 |H'|dvdv_*d\sigma dx
 \\&=\iiiint_{|v-v_*|\ge 1}+\iiiint_{|v-v_*|< 1}=A_{221}+A_{222}.
 \end{align*}
The integral $A_{221}$, since   $|v-v_*|^{\gamma+\lambda}\le 1$ holds,
can be reduced to a special case of (1) with $\gamma+\lambda=0$, implying
\begin{align*}
A_{221}&\lesssim
\iiiint_{\RR^9\times\SS^2}\theta^{-2-2s+\lambda}
|F_*||(W_{\ell} G)|
 |H'|dvdv_*d\sigma dx
 \\&\lesssim\int \| F\|_{L^1_v}\Big\{\iint_{\RR^3\times\SS^2}\theta^{-2-2s+\lambda}
\ |W_{\ell} G|
 |H'|dvd\sigma\Big\} dx
 \\&\lesssim
\int
\|W_{\ell_1}\phi_2(x-a)\partial^{\alpha_2}_{\beta_2}
 g^\varepsilon\|_{L^2_v}
\|h^{\alpha_1,\beta_1}_{\ell}\|_{L^2_v}\|h^{\alpha,\beta}_{\ell}\|_{L^2_v}dx
\\&\lesssim
\int
\|W_{\ell_1}\phi_2(x-a)\partial^{\alpha_2}_{\beta_2}
 g^\varepsilon\|_{L^2_v}
\|h^{\alpha_1,\beta_1}_{\ell}\|_{L^2_v}\|h^{\alpha,\beta}_{\ell+1}\|_{L^2_v}dx
 \end{align*}
 for $\ell_1>3/2$,
 and hence by the Sobolev embedding,
\[
A_{221}
\lesssim \|g^\varepsilon_{\ell}\|_{H^{k,\ell}_{\it ul}}^2
\|h^{\alpha, \beta}_{\ell+1}\|
.
\]
On the other hand
\begin{align*}
A_{222}&\lesssim \iiiint_{\RR^9\times\SS^2, |v-v_*|<1}\theta^{-2-2s+\lambda}|v-v_*|^{\gamma+\lambda}
|F_*||(W_{\ell} G)|
 |H'|dvdv_*d\sigma dx
 \\&\lesssim
 \iiiint_{\RR^6\times\SS^2}\theta^{-2-2s+\lambda}
\Big\{\int_{\RR^3, |v-v_*|<1}|v-v_*|^{\gamma+\lambda}
|F_*|dv_*\Big\}
|(W_{\ell} G)|
 |H'|dvd\sigma dx.
\end{align*}
Clearly, if  $\gamma+\lambda>-3/2$, by the Cauchy-Schwarz inequality,
\begin{align*}
\int_{\RR^3, |v-v_*|<1}&|v-v_*|^{\gamma+\lambda}
|F_*|dv_*
\\&\le \Big(\int_{\RR^3, |v-v_*|<1}|v-v_*|^{2(\gamma+\lambda)}dv_*\Big)^{1/2}
\|F\|_{L^1_v}
\lesssim  \|F\|_{L^1_v},
\end{align*}
so that we get again by the Sobolev embedding

\begin{align*}
 A_{222}&
 \lesssim
  \iiiint_{\RR^6\times\SS^2}\theta^{-2-2s+\lambda}
 |F|_{L^1_v}
 |W_{\ell} G|
  |H'|dvd\sigma dx
\\&\lesssim
  \int |W_{\ell_1}F|_{L^2_v}
 |W_{\ell} G|_{L^2_v}
  |H|_{L^2_v}dx
  \\&\le
\int
\|W_{\ell_1}\phi_2(x-a)\partial^{\alpha_2}_{\beta_2}
 g^\varepsilon\|_{L^2_v}
\|h^{\alpha_1,\beta_1}_{\ell}\|_{L^2_v}\|h^{\alpha,\beta}_{\ell}\|_{L^2_v}dx
\\&\le
\int
\|W_{\ell_1}\phi_2(x-a)\partial^{\alpha_2}_{\beta_2}
 g^\varepsilon\|_{L^2_v}
\|h^{\alpha_1,\beta_1}_{\ell}\|_{L^2_v}\|h^{\alpha,\beta}_{\ell+1}\|_{L^2_v}dx
\\&
\lesssim \|g^\varepsilon_{\ell}\|_{H^{k,\ell}_{\it ul}}^2
\|h^{\alpha, \beta}_{\ell+1}\|
.
\end{align*}
Consequently, we have for both positive and negative  $\gamma+\lambda$
\[
A_{22}\lesssim \|g^\varepsilon_{\ell}\|_{H^{k,\ell}_{\it ul}}^2
\sum_{|\alpha'+ \beta'|\le k}\|h^{\alpha', \beta'}_{\ell+1}\|.
\]
Combine the above estimates to conclude the proof of Lemma \ref{lem3.4}.
\end{proof}
This completes the proof of the part of Lemma \ref{lem410} for $\Xi_1$.

\subsection{Estimate of $\Xi_2$.} Since $|\nabla \phi_1(x)|\le C\phi_2(x)$, we have
\begin{align*}
\Xi_2&=[ \phi_1(x-a)W_\ell\partial^\alpha_\beta, v\cdot\nabla_x]g^\varepsilon
\\&=\phi_1(x-a)W_{\ell}\nabla_x\pa^\alpha_{\beta-1}g^\varepsilon-
(v\cdot\nabla_x\phi_1(x-a))||W_{\ell}\pa^\alpha_{\beta}g^\varepsilon,
\\|\Xi_2|&\le\phi_1(x-a)W_{\ell}|\nabla_x\pa^\alpha_{\beta-1}g^\varepsilon|
+\phi_2(x-a)W_{\ell+1}|\pa^\alpha_{\beta}g^\varepsilon|,
\end{align*}
whence
\begin{align*}
|(\Xi_2, h^{\alpha,\beta}_\ell)|&\le( \|\phi_1(x-a)W_{\ell-1}\nabla_x\pa^\alpha_{\beta-1}g^\varepsilon\|+
 \|\phi_2(x-a)W_{\ell}\pa^\alpha_{\beta}g^\varepsilon\|)\|h^{\alpha,\beta}_{\ell+1}\|
 \\&\lesssim
 \|g^\varepsilon\|_{H^{k,\ell}_{\it ul}}\|h^{\alpha,\beta}_{\ell+1}\|,
\end{align*}
which implies Lemma \ref{lem410} for $\Xi_2$.

\subsection{Estimate of $\Xi_3$.} It holds that
\begin{align*}
|\Xi_3|&\le |
\kappa[ \phi_1(x-a)W_\ell\partial^\alpha_\beta, \langle v\rangle^2]g^\varepsilon|
\\&\lesssim
\kappa\phi_1(x-a)W_\ell(|v||\pa^{\alpha}_{\beta-1}g^\varepsilon|
+|\pa^{\alpha}_{\beta-2}g^\varepsilon|
\end{align*}
whence follows
\begin{align*}
|(\Xi_3, h^{\alpha,\beta}_{\ell})|&\lesssim \kappa
(\|\phi_1(x-a)W_{\ell-1}|v||\pa^{\alpha}_{\beta-1}g^\varepsilon\|
+\|\phi_1(x-a)W_{\ell-1}\pa^{\alpha}_{\beta-2}g^\varepsilon\|)\
\|h^{\alpha,\beta}_{\ell+1}\|
\\&\lesssim
\kappa\|g^\varepsilon\|_{H^{k,\ell}_{\it ul}}\|h^{\alpha,\beta}_{\ell+1}\|.
\end{align*}
This proves the part of Lemma \ref{lem410} for $\Xi_3$, and hence we are done.

\bigskip
\noindent
{\bf Acknowledgements :}
The research of the first author was supported in part by Zhiyuan foundation and Shanghai Jiao Tong University. The research of the second author was
supported by  Grant-in-Aid for Scientific Research No.22540187,
Japan Society of the Promotion of Science.
The last author's research was supported by the General Research
Fund of Hong Kong,
 CityU No.103109, and the Lou Jia Shan Scholarship programme of
Wuhan University.
 Finally the authors would like to
thank the financial support of City University of Hong Kong, Kyoto
University and Wuhan University during each of their stays, mainly starting {}from 2006. These supports have enable the final conclusion through our previous papers.

\end{document}